\documentclass[eps,preprint]{imsart}

\usepackage{amsthm,amsmath}
\RequirePackage[numbers]{natbib}
\RequirePackage[colorlinks,citecolor=blue,urlcolor=blue]{hyperref}
\usepackage[psamsfonts]{amssymb}
\usepackage{graphicx}
\usepackage{subfigure}
\usepackage{threeparttable}

\doi{10.1214/154957804100000000}
\pubyear{0000}
\volume{0}
\firstpage{0}
\lastpage{0}
\arxiv{1306.5505}

\startlocaldefs

\newtheorem{definition}{Definition}
\newtheorem{theorem}{Theorem}
\newtheorem{corollary}{Corollary}
\newtheorem{lemma}{Lemma}

\endlocaldefs

\begin{document}

\begin{frontmatter}

\title{Asymptotic Properties of Lasso+mLS and Lasso+Ridge in Sparse High-dimensional Linear Regression}
\runtitle{Lasso+mLS and Lasso+Ridge in SHLR}


\author{\fnms{Hanzhong} \snm{Liu}\ead[label=e1]{lhz2009@pku.edu.cn}}
\address{School of Mathematical Sciences, \\ Peking University, Beijing 100871, P.R. China \\ \printead{e1}}
\and
\author{\fnms{Bin} \snm{Yu}\ead[label=e2]{binyu@stat.berkeley.edu}}
\address{Department of Statistics, \\ Department of Electrical Engineering and Computer Sciences, \\ University of California, Berkeley, CA 94720, USA \\ \printead{e2}}

\runauthor{Hanzhong Liu and Bin Yu}

\begin{abstract}

We study the asymptotic properties of Lasso+mLS and Lasso+Ridge under the sparse high-dimensional linear regression model: Lasso selecting predictors and then modified Least Squares (mLS) or Ridge estimating their coefficients. First, we propose a valid inference procedure for parameter estimation based on parametric residual bootstrap after Lasso+mLS and Lasso+Ridge. Second, we derive the asymptotic unbiasedness of Lasso+mLS and Lasso+Ridge. More specifically, we show that their biases decay at an exponential rate and they can achieve the oracle convergence rate of $s/n$ (where $s$ is the number of nonzero regression coefficients and $n$ is the sample size) for mean squared error (MSE). Third, we show that Lasso+mLS and Lasso+Ridge are asymptotically normal. They have an oracle property in the sense that they can select the true predictors with probability converging to $1$ and the estimates of nonzero parameters have the same asymptotic normal distribution that they would have if the zero parameters were known in advance. In fact, our analysis is not limited to adopting Lasso in the selection stage, but is applicable to any other model selection criteria with exponentially decay rates of the probability of selecting wrong models.

\end{abstract}



\received{\smonth{6} \syear{2013}}

\end{frontmatter}

\section{Introduction}

Consider the sparse linear regression model
\begin{equation}
 Y = X \beta^* +\epsilon,
 \label{eqn:lrm}
\end{equation}
where $\epsilon = (\epsilon_1,...,\epsilon_n)^T$ is a vector of independent and identically distributed (i.i.d.) random variables with mean 0 and variance $\sigma^2$. $Y=(y_1,...,y_n)^T \in \mathbb R^n$ is the response vector, and $X \in \mathbb R^{n \times p}$ is the design matrix which is deterministic. $\beta^*\in \mathbb R^p$ is the vector of model coefficients with at most $s$ $(s<n)$ non-zero components. We consider the high-dimensional setting which allows $p$ and $s$ to grow with $n$ ($p$ can be comparable to or larger than $n$). Note that, in here and what follows, $Y$, $X$, and $\beta^*$ are all indexed by the sample size $n$, but we omit the index whenever this does not cause confusion.

In sparse linear regression models, an active line of research focuses on the recovery of sparse vector $\beta^*$ by a popular $l_1$ regularization method called Lasso \cite{Tibshirani1996}. The Lasso has been studied under at least three common criteria: (i) model selection criteria, meaning the correct recovery of the support set $S=\{ j \in \{1,2,...,p\}: \beta^*_j \neq 0\}$ of the model coefficients $\beta^*$; (ii) $l_q$ estimation errors $||\hat \beta - \beta^*||_q^q$, especially $l_2$ and $l_1$, where $\hat \beta$ is the estimate of $\beta^*$; and (iii) prediction error $||X\hat \beta - X\beta||_2^2$.

The Lasso estimator is defined by
\begin{equation}
 \hat \beta(\lambda_n) = \mathop {argmin}\limits_{\beta} \left\{  || {Y - X\beta } ||_2^2  + \lambda_n || {\beta } ||_1 \right\},
 \label{eqn:lasso}
\end{equation}
where $\lambda_n \geq 0$ is the tuning parameter which controls the amount of regularization applied to the estimate. Setting $\lambda_n =0$ reverses the Lasso problem to Ordinary Least Squares (OLS) which minimizes the unregularized empirical loss.

Replacing $\ell_1$ penalty by $\ell_2$ penalty in (\ref{eqn:lasso}) gives the Ridge estimator \cite{Hoerl1970}:
\begin{equation}
 \hat \beta_{Ridge}(\lambda_n) = \mathop {argmin}\limits_{\beta} \left\{  || {Y - X\beta } ||_2^2  + \lambda_n || {\beta } ||_2^2 \right\},
 \label{eqn:Ridge}
\end{equation}

The Lasso estimator has two nice properties, namely, (i) it generates sparse models by means of $\ell_1$ regularization and (ii) it is also computationally feasible (see \cite{OsbornePresnell2000,EfronHastieTibshirani2004,FriedmanHastieHolfing2007}). The asymptotic behavior of Lasso-type estimators has been studied by \cite{KnightFu2000} for fixed $p$ and $\beta^*$ as $n\rightarrow \infty$. In particular, they have shown that under some regularity conditions on the design, $\lambda_n=o(n)$ is sufficient for consistency in the sense that $\hat \beta(\lambda_n) \rightarrow_p \beta^*$, and $\lambda_n$ should grow more slowly (i.e. $\lambda_n=O(\sqrt{n})$) for asymptotic normality of the Lasso estimator. On the model selection consistency front, \cite{MeinshausenBuhlmann2006} proposed the neighborhood stability condition which is equivalent to the Irrepresentable condition \cite{Fuchs2005,Tropp,ZhaoYu2006,Wainwright2009} to prove the Lasso consistency for Gaussian graphical model selection. \cite{ZhaoYu2006} showed that the Irrepresentable condition is almost necessary (for fixed $p$) and sufficient for the Lasso to select the true model both in the classical fixed $p$ setting and in the high-dimensional setting. \cite{ZhangHuang2008} considered a weaker sparsity assumption, meaning that the regression coefficients outside an ideal model are small but not necessarily zero (the sum of their absolute values is of the order $O(s\lambda_n/n)$), and imposed the sparse Riesz condition to prove the rate-consistent in terms of the sparsity, bias and the norm of missing large coefficients. \cite{Wainwright2009} further established precise conditions on the scalings of $(n,p,s)$ that are necessary and sufficient for sparsity pattern recovery using the Lasso. In addition, thresholded Lasso and Dantzig estimators were introduced in \cite{Karim2008} and the authors proved their model selection consistency under less restrictive conditions on the decay rates of the nonzero regression coefficients. Other related work includes \cite{DonohoElad2000,Zou2006,Bunea2008,Adel2013}. All the aforementioned papers imposed suitable mutual incoherence conditions on the design. On the $l_2$ estimation error front, the Lasso has been shown under a weaker restricted eigenvalue condition to achieve $l_2$ convergence rate of $(s\log p) / n$ \cite{vandeGeer2007,BickelRitov2009,MeinshausenandYu2009,unifiedpaper2009,unifiedpaper2009journal}, which is the minimax optimal rate \cite{Raskutti2011}. Other work focuses on the convergence rates of $||X\hat \beta(\lambda_n)- X\beta^*||_2^2$ and $||\hat \beta(\lambda_n) - \beta^*||_1$, see \cite{Greenshtein2004,vandeGeerprediction,Bunea2009} for example.

However, even if $p$ is fixed and the Irrepresentable Condition is satisfied, there does not exist a tuning parameter $\lambda_n$ which can lead to both variable selection consistency and asymptotic normality \cite{FanLi2001,Zou2006}. More importantly, for the case of $p\gg n$, statistical inference for the Lasso estimator with theoretical guarantees is still an insufficiently explored area.

The bootstrap is very useful for inference. For fixed $p$, \cite{Minnier2009} developed a perturbation resampling-based method to approximate the distribution of a general class of penalized regression estimates. \cite{Chatterjee2011} proposed a modified residual bootstrapping Lasso method that is consistent in estimating the limiting distribution of the Lasso estimator. \cite{cuihuizhang2011} developed a low-dimensional projection (LDP) approach to constructing confidence intervals. Though LDP works for $(s \log{p})/\sqrt{n}\rightarrow 0$, it has nothing to do with the idea of resampling and bootstrap. Does the bootstrap provide a valid approximation in the case of $p\gg n$? In this paper, we will give an affirmative answer to this question based on residual bootstrap after two post-Lasso estimators: Lasso+mLS and Lasso+Ridge. Our method provides consistent estimate of the limiting distribution of Lasso+mLS (or Lasso+Ridge) even if $p$ grows at an exponential rate in $n$.

Post-Lasso estimator is a special case of two stage estimators: (1) selection stage: one selects predictors using the Lasso; and (2) estimation stage: modified Least Square (mLS) or Ridge, is applied to estimate the coefficients of the selected predictors. Our estimator is referred to as Lasso+mLS or Lasso+Ridge. Lasso+mLS is very close to Lasso+OLS \cite{Belloni2009}, which uses Ordinary Least Squares (OLS) in the second stage. Several authors have previously considered two stage estimators to improve the performance of the Lasso, such as the Lars-OLS hybrid \cite{EfronHastieTibshirani2004}, adaptive Lasso \cite{Zou2006}, relaxed Lasso \cite{Meinshausen2007}, and marginal bridge estimator \cite{Huang2008}, to name just a few.

{\bf Our contributions} are summarized as follows:

\begin{enumerate}
  \item We propose a valid inference procedure for parameter estimation based on parametric residual bootstrap after two post-Lasso estimators: Lasso+mLS and Lasso+Ridge. More specifically, we show that the Mallows distance between the distributions of the bootstrap estimator and the Lasso+mLS (or Lasso+Ridge) estimator converges to 0 in probability.
  \item Under the Irrepresentable condition and other regularity conditions, we derive the asymptotic unbiasedness of Lasso+mLS and Lasso+Ridge. We show that their biases decay at an exponential rate and that they can achieve the oracle convergence rate of $s/n$ for mean squared error $E||\tilde \beta- \beta^*||_2^2$ where $\tilde \beta$ is either Lasso+mLS or Lasso+Ridge.
  \item We prove the asymptotic normality of Lasso+mLS and Lasso+Ridge. As we show in Theorem~\ref{theorem: asymptotic normality} and Corollary~\ref{corollary: asymptotic normality}, these two post-Lasso estimators display an oracle property that the Lasso does not have: they can select the true predictors with probability converging to 1 and the estimates of nonzero parameters have the same asymptotic normal distribution that they would have if the zero parameters were known in advance.
  \item Our analysis is not limited to adopting the Lasso in the selection stage, but is applicable to any other model selection criteria with exponentially decay rates of the probability of selecting wrong models, for example, stability selection \cite{MeinshausenBuhlmann2010}, SCAD \cite{FanLi2001,SCAD2009} and Dantzig selector \cite{CandesTao2007,BickelRitov2009,Dantiz2013}.
\end{enumerate}

Our key assumptions for the validity of residual bootstrap after Lasso+mLS or Lasso+Ridge are the Irrepresentable condition and that $s$ goes to infinity slower than $\sqrt{n}$. The Irrepresentable condition can be weakened by the sparse eigenvalue condition if we adopt stability selection \cite{MeinshausenBuhlmann2010} to enhance the selection performance of the Lasso. Without considering model selection, \cite{BickelFreedman1983} showed that residual bootstrap OLS fails if $p^2/n$ does not tend to 0. Therefore, the condition $s^2/n \rightarrow 0$ cannot be weakened. Our conditions on the scalings $(n,p,s)$ are not the sharpest but have been previously used in the literature \cite{ZhaoYu2006} and make our convergence rate more explicit. In addition, we require a gap of size $n^{\frac{c_3}{2}}$ ($c_3 \in (0,1]$ is a constant) between the decay rate of $\beta^*$ and $n^{-\frac{1}{2}}$ which prevents the estimation from being dominated by the noise terms. \cite{FanLv2008} proposed a similar constraint $\mathop {min}\limits_{1\leq i \leq s}|\beta_i^*| \geq \frac{M}{n^\kappa},\ 0 \leq \kappa < \frac{1}{2}$ to show model selection consistency of the Sure Independent Screening. This assumption is weaker than $\mathop {min}\limits_{1\leq i \leq s}|\beta_i^*| \geq M$, which was assumed by \cite{Huang2008} in connection with the asymptotic properties of the Bridge estimator. However, as mentioned by \cite{Leeb2005}, inference results based on post-model selection methods can be misleading when this kind of ``beta min" condition fails. Therefore, the proposed inference procedure should be used in practice only when there is believed to be a gap between the decay rate of the nonzero elements of $\beta^*$ and the $n^{-1/2}$. Finally, we need some regularity conditions (conditions (a)-(c) in Section 2) which are standard in sparse high-dimensional linear regression literature \cite{ZhaoYu2006,Huang2008,HuangZhang2008}.

After we had obtained our bootstrap results in Theorem~\ref{theorem:bootstraplassovalid} and Corollary~\ref{corollary:bootstraplassovalid}, our attention was brought to an independent result in \cite{Chatterjee2012} where a variant of the Irrepresentable condition was used to prove the second-order correctness of the residual bootstrap applied to a suitable studentized version of the adaptive Lasso estimator. However, the main results of \cite{Chatterjee2012} are valid only for linear combinations of the adaptive Lasso estimator while our results hold for the joint distribution of Lasso+mLS (or Lasso+Ridge). The distance between distributions used in \cite{Chatterjee2012} and the proof there are also different from ours. Specifically, \cite{Chatterjee2012} adopted the total variation distance and used the Edgeworth expansion in the proof while we study the Mallows distance and our proof is direct. In addition, \cite{Chatterjee2012} allows $p$ to grow only at polynomial rates in $n$ while we allow $p$ to grow at an exponential rate in $n$.

In our paper, before stating the bootstrap result, we first derive the asymptotic unbiasedness and asymptotic normality of Lasso+mLS and Lasso+Ridge. It is well known that $(s\log p)/n$ is the minimax optimal rate for the Lasso under the restricted eigenvalue condition. Since we assume the stronger Irrepresentable condition and some conditions on scaling $(n,p,s)$, we are able to attain a better rate of $s/n$ for MSE which indicates that one can avoid the feature selection penalty of $\log p$ by combining the Lasso and Least Squares or Ridge. We should mention that previous work \cite{Belloni2009} has obtained $l_2$ convergence rate $(||\tilde \beta_{Lasso+OLS} - \beta^*||_2^2=O_p(s/n))$ of Lasso+OLS estimator under weaker conditions. However, their results hold in probability and it is not clear whether Lasso+OLS can achieve the oracle convergence rate of $O(s/n)$ in $L_2$-expectation, i.e., whether $E||\tilde \beta - \beta^*||_2^2=O(s/n)$ holds, which we need to prove the validity of residual bootstrap. On the asymptotic normality front, the authors in \cite{Belloni2011a,Belloni2011b} also adopted the OLS after model selection and derived the asymptotic normality for inference on the effect of a treatment variable on a scalar outcome in the presence of very many controls. However, they studied a partial linear model which is different from ours and the $l_1$ regularization was imposed on the effect of the control variables without on the effect of the treatment variable.

{\bf Notation} For any vector $a=(a_1,...,a_m)^T$, we denote $||a||_2^2 = \sum_{i=1}^{m} a_i^2$, $||a||_1 = \sum_{i=1}^{m} |a_i|$, and $\|a\|_\infty = \max_{i=1,\ldots,m} |a_i|$. For a vector $\beta \in R^p$ and a set $S \subset \{1,...,p\}$, denote $S^c$ the complementary set of $S$ and let $\beta_S = \{\beta_j: j \in S \}$. Given an $n$ by $p$ matrix $X$, write $x_i^T \in R^p,i=1,...,n$ and $X_j \in R^n,j=1,...,p$ the $i$-th row and the $j$-th column of $X$ respectively, where $x_i^T$ is the transpose of $x_i$. For a given $m\times m$ matrix $A$, let $\Lambda_{min} (A)$ and $\Lambda_{max} (A)$ denote the smallest and largest eigenvalues of $A$ respectively. Write $tr(A)$ the trace of $A$ which is the sum of the diagonal entries of $A$.

The rest of the paper is organized as follows: in Section 2, we define modified Least Squares and Ridge after model selection and study their asymptotic properties. In Section 3, we apply these general properties to the special cases of Lasso+mLS and Lasso+Ridge and then derive their asymptotic unbiasedness, asymptotic normality and the approximation property of residual bootstrap. Similar asymptotic properties of modified Least Squares and Ridge after stability selection are obtained in Section 4. Simulation examples are given in Section 5. We conclude in Section 6. The proofs can be found in the Appendix.

\section{Asymptotic Properties of Modified Least Squares and Ridge after Model Selection}

In this section, we begin with a precise definition of the modified Least Squares or Ridge after model selection, and then study their asymptotic properties, including asymptotic unbiasedness, asymptotic normality and the validity of residual bootstrap.

\subsection{Definitions and Assumptions}

Modified Least Squares (or Ridge) after model selection refers to a special type of two stage estimators. In the first stage, one uses certain model selection methods to select predictors. For example, let $\hat \beta$ be the Lasso estimator defined in $\eqref{eqn:lasso}$, one gets a set of selected predictors $\hat S = \{ j \in \{1,2,...,p\}:\ \hat \beta_j \neq 0 \}$. Again, $\hat \beta$ and $\hat S$ are dependent on $\lambda_n$, but we omit the dependence whenever this does not cause confusion.

In the second stage, a low-dimensional estimation method is applied to the selected predictors in $\hat S$. For example, one can adopt OLS and then form the OLS after model selection (denoted by Select+OLS):
\begin{equation}
 \tilde \beta_{Select+OLS} = \mathop {argmin}\limits_{\beta:\ \beta_j = 0,\ j \in \hat S^c}  || {Y - X\beta } ||_2^2.
 \label{eqn:ols}
\end{equation}
The solution of $\eqref{eqn:ols}$ is $\tilde \beta_{Select+OLS,\hat S} = (X_{\hat S}^TX_{\hat S})^{-1}X_{\hat S}^TY$ if $X_{\hat S}^TX_{\hat S}$ is invertible. When $X_{\hat S}^TX_{\hat S}$ is not invertible, the solution of $\eqref{eqn:ols}$ is not unique. In this case one can use the generalized inverse. However, the generalized inverse is not stable when the smallest nonzero eigenvalue of $X_{\hat S}^TX_{\hat S}$ approximately equals 0, which may result in poor performance. We propose a modified Least Squares method in the second stage and form our Select+mLS estimator. Let $d = |\hat S|$ and write $\frac{1}{\sqrt{n}} X_{\hat S}$ in its singular value decomposition (SVD) form
\begin{equation}
 \frac{1}{\sqrt{n}} X_{\hat S} = U D V^T
 \label{eqn:SVD}
\end{equation}
where $U$ is an $n\times n$ orthogonal matrix, $D$ is an $n\times d$ diagonal matrix with singular values $\lambda_1 \geq \lambda_2 \geq ... \geq \lambda_d$ on the diagonal, and $V^T$ (the transpose of $V$) is a $d\times d$ orthogonal matrix. By simple algebraic operations, OLS based on $( X_{\hat S},Y)$ has the following form:
\begin{equation}
 \tilde \beta_{OLS} = \frac{1}{\sqrt{n}} V D^{-1} U^T Y
 \label{eqn:ols2}
\end{equation}
where $D^{-1}$ is a $d\times n$ diagonal matrix with diagonal entries $\lambda_1^{-1}$, $\lambda_2^{-1}$,...,$\lambda_d^{-1}$. If one or more of the singular values are 0, one can utilize generalized inverse which just takes $\lambda_k^{-1}=0$ for all zero-valued $\lambda_k$.

We propose a hard thresholding on the singular values, that is, shrinking those singular values smaller than $\tau_n$ $(\tau_n>0)$ to zero. Then define a modified Least Square estimator $\tilde \beta_{mLS}(\tau_n)$ in the same form of \eqref{eqn:ols2} except that we take $\lambda_k^{-1}=0$ for all $\lambda_k< \tau_n$. This estimator is similar to principal components regression \cite{Massy1965}. Note that if $0\leq \tau_n^2 \leq \Lambda_{min}(\frac{1}{n}X_{\hat S}^TX_{\hat S})$, $\tilde \beta_{mLS}(\tau_n)$ is the same as $\tilde \beta_{OLS}$, that is,
\[ \tilde \beta_{mLS}(\tau_n) =  \tilde \beta_{OLS} = (X_{\hat S}^TX_{\hat S})^{-1}X_{\hat S}^TY, \ \  if \ \  0\leq \tau_n^2 \leq \Lambda_{min}(\frac{1}{n}X_{\hat S}^TX_{\hat S}). \]

Our final modified Least Squares after model selection (Select+mLS) is defined by:
\begin{equation}
 \tilde \beta_{Select+mLS,\hat S}(\tau_n) = \tilde \beta_{mLS}(\tau_n)\ , \ \  \tilde \beta_{Select+mLS,\hat S^c}(\tau_n)=0.
 \label{eqn:modifiedols}
\end{equation}

One can also use Ridge in the second stage and form the Ridge after model selection (Select+Ridge):
\begin{equation}
 \tilde \beta_{Select+Ridge}(\mu_n) = \mathop {argmin}\limits_{\beta:\ \beta_j = 0,\ j \in \hat S^c}  || {Y - X\beta } ||_2^2 + \mu_n ||\beta||_2^2
 \label{eqn:ridge}
\end{equation}
where $\mu_n \geq 0$ is a smoothing parameter. \eqref{eqn:ridge} is equivalent to
\begin{equation}
 \tilde \beta_{Select+Ridge,\hat S}(\mu_n) =(X_{\hat S}^TX_{\hat S}+ \mu_n I)^{-1}X_{\hat S}^TY\ , \ \  \tilde \beta_{Select+Ridge,\hat S^c}(\mu_n)=0.
 \label{eqn:lassoridge}
\end{equation}

When the Lasso is used in the selection stage, we refer to our final estimators as Lasso+mLS and Lasso+Ridge respectively. $\tau_n$ and $\mu_n$ are tuning parameters. In our theorems and simulation, $\tau_n \propto \frac{1}{n}$ and $\mu_n \propto \frac{1}{n}$ can get good estimation and prediction performance. For the sake of notational simplicity, we omit the dependence of estimators on $\lambda_n$ and $\tau_n$ or $\mu_n$ whenever this does not cause confusion.

To state our main theorems, we need the following assumptions. Without loss of generality, assume $\beta^*=(\beta^*_1,...,\beta^*_s,\beta^*_{s+1},...,\beta^*_p)$ with $\beta^*_j\neq 0$ for $j=1,...,s$ and $\beta^*_j = 0$ for $j=s+1,...,p$. Let $S=\{1,...,s\}$ and $\beta^*_S = (\beta^*_1,...,\beta^*_s)$. Now write $X_S$ and $X_{S^c}$ as the first $s$ and the last $p-s$ columns of $X$ respectively and let $C=\frac{1}{n} X^TX$ which can be expressed in a block-wise form as follows:
\begin{equation}
\label{eqn:cform}
C = \left(
  \begin{array}{cc}
    C_{11} & C_{12} \\
    C_{21} & C_{22} \\
  \end{array}
\right)
\end{equation}
where $C_{11}=\frac{1}{n}X_S^TX_S$, $C_{12}=\frac{1}{n}X_S^TX_{S^c}$, $C_{21}=\frac{1}{n}X_{S^c}^TX_S$ and $C_{22}=\frac{1}{n}X_{S^c}^TX_{S^c}$.
\ \\

{\em Assumption (a): $\epsilon_i$ are i.i.d. gaussian random variables with mean 0 and variance $\sigma^2$.}

{\em Assumption (b): Suppose that the predictors are standardized, i.e.
\begin{equation}
 \frac{1}{n} \sum\limits_{i=1}^{n} {x_{ij}} = 0 \ and \ \frac{1}{n} \sum\limits_{i=1}^{n} {x_{ij}^2}=1, \ j=1,...,p.
 \label{eqn:standardized}
\end{equation}
}

{\em Assumption (c): there exists an constant $\Lambda_{min}>0$ such that
\begin{equation}
 \Lambda_{min} (C_{11})  \geq \Lambda_{min}.
 \label{eqn:cond2}
\end{equation}
}

{\em Assumption (d): the model is high-dimensional and sparse, i.e. there exists $0\leq c_1<1$ and $0<c_2<1-c_1$ such that
\begin{equation}
 s=s_n=O(n^{c_1}) \ , \ p=p_n = O(e^{n^{c_2}}).
 \label{eqn:cond3}
\end{equation}
}

{\em Assumption (e)\footnote{In fact, Theorem~\ref{theorem: asymptotic normality} (asymptotic normality) and Theorem~\ref{theorem:bootstraplassovalid} (bootstrap) are valid for any $\mu_n \rightarrow 0$ with rate neither faster than $e^{-n^{c_2}/4}$ nor slower than $\frac{1}{n}$.}: $\tau_n \propto \frac{1}{n}$ and $\mu_n \propto \frac{1}{n}$.
}
\ \\

The gaussian assumption (a) is fairly standard in the literature. Assumption (c) ensures the smallest eigenvalue of $C_{11}$ is bounded away from 0 so that its inverse behaves well. (a)-(c) are typical assumptions in sparse linear regression literature, see for example \cite{ZhaoYu2006,Huang2008,HuangZhang2008}. Assumption (d) means that the number of relevant predictors $s$ is allowed to diverge but much slower than $n$, and that the number of predictors $p$ can grow faster than $n$ (up to exponentially fast), which is standard in almost all high-dimensional inference literature. Though this assumption is stronger than the typical one $\frac{s \log p}{n} \rightarrow 0$, it has been used previously \cite{ZhaoYu2006}.

In the following subsections, we will show that both Select+mLS and Select+Ridge have good asymptotic properties if the probability of selecting wrong models $P(\hat S \neq S)$ decays fast, say $o(e^{-n^{\kappa}})$ (where $\kappa>0$ is a constant).

\subsection{Bias and MSE of Select+mLS and Select+Ridge}

In a high-dimensional setting, bias is not the only consideration of estimates because of the bias and variance trade-off. Regularization has been a popular technique for model fitting which results in a biased estimator but decreases the mean squared error (MSE) dramatically. However, if two estimators have the same MSE, we prefer the unbiased one. The following Theorems~\ref{theorem: asymptotic bias and variance} and \ref{theorem: asymptotic bias and variance ridge} provide general bounds for the bias and MSE of the Select+mLS and Select+Ridge.

\begin{theorem}
\label{theorem: asymptotic bias and variance}
Suppose that Gaussian assumption (a) is satisfied and $\tau_n^2 \leq \Lambda_{min}(C_{11})$, then the bias and the MSE of $\tilde \beta_{Select+mLS}(\tau_n)$ satisfy
\begin{eqnarray}
 &  & || E \tilde \beta_{Select+mLS}(\tau_n) - \beta^*||_2^2  \nonumber \\
 & \leq & 2 P(\hat S \neq S)  \left\{\frac{\sigma^2}{n} tr(C_{11}^{-1}) + ||\beta^*||_2^2 + \frac{1}{\tau_n^2} \frac{1}{n}||X\beta^*||_2^2 + \frac{1}{\tau_n^2} \sigma^2 \right\},
 \label{eqn:theo1}
\end{eqnarray}
\begin{eqnarray}
 & & E||\tilde \beta_{Select+mLS}(\tau_n) - \beta^*||_2^2 \nonumber \\
 & \leq & \frac{\sigma^2}{n} tr(C_{11}^{-1}) + 8 \sqrt{ P(\hat S \neq S) }  \left\{ ||\beta^*||_2^2 + \frac{1}{\tau_n^2}\frac{1}{n}||X\beta^*||_2^2  + \frac{1}{\tau_n^2}\sigma^2 \right\}.
 \label{eqn:theo2}
\end{eqnarray}
\end{theorem}
\ \\
{\bf Remark 2.1.} Let $\beta_S^{OLS}=(X_S^TX_S)^{-1}X_S^TY = \beta^*_S+(X_S^TX_S)^{-1}X_S^T\epsilon$ be the oracle OLS estimator. It is easy to see that $E||\beta_S^{OLS} - \beta^*||_2^2  =  \frac{\sigma^2}{n} tr(C_{11}^{-1})$. The first term on the right hand side of $\eqref{eqn:theo2}$ corresponds to the oracle convergence rate. The second term is related to model selection accuracy. In the case of the Lasso, $P(\hat S \neq S)$ can decay at an exponential rate. Hence the MSE is completely determined by the first term $\frac{\sigma^2}{n} tr(C_{11}^{-1})$, which can not be improved.
\ \\
{\bf Remark 2.2.} From theorem~\ref{theorem: asymptotic bias and variance}, one can easily get an upper bound for prediction mean squared error $E\{\frac{1}{n}||X\tilde \beta -X\beta^*||_2^2\}$ since
\[ \frac{1}{n}||X\tilde \beta -X\beta^*||_2^2 \leq \Lambda_{max}(\frac{1}{n}X^TX) ||\tilde \beta - \beta^*||_2^2. \]

Similarly, for $\tilde \beta_{Select+Ridge}(\mu_n)$, we have:
\begin{theorem}
\label{theorem: asymptotic bias and variance ridge}
Suppose that Gaussian assumption (a) is satisfied, then the bias and the MSE of $\tilde \beta_{Select+Ridge}(\mu_n)$ satisfy
\begin{eqnarray}
 & & || E \tilde \beta_{Select+Ridge}(\mu_n) - \beta^*||_2^2 \nonumber \\
 & \leq & \frac{2\mu_n^2}{n^2 \Lambda_{min}^2} ||\beta^*||_2^2 +  2 P(\hat S \neq S) \frac{n}{\mu_n} \left\{ \frac{1}{n}||X\beta^*||_2^2  + \sigma^2 \right\},
 \label{eqn:theo21}
\end{eqnarray}
\begin{eqnarray}
 & & E||\tilde \beta_{Select+Ridge}(\mu_n) - \beta^*||_2^2  \nonumber \\
 & \leq & \frac{\sigma^2}{n} tr\{(C_{11}+\frac{\mu_n}{n}I)^{-2}C_{11}\} + \frac{\mu_n^2}{n^2 \Lambda_{min}^2} ||\beta^*||_2^2 +  \nonumber \\
  & &  2 \sqrt{ P(\hat S \neq S) }  \left\{ ||\beta^*||_2^2 + \frac{n}{\mu_n} \{\frac{1}{n}||X\beta^*||_2^2  + \sigma^2\}  \right\}.
 \label{eqn:theo22}
\end{eqnarray}
\end{theorem}

Theorems~\ref{theorem: asymptotic bias and variance} and \ref{theorem: asymptotic bias and variance ridge} indicate that as long as the probability of selecting wrong models $P(\hat S \neq S)$ decays fast, say at an exponential rate, Select+mLS and Select+Ridge are asymptotically unbiased and their MSEs decay at the oracle rate. We have known that under the Irrepresentable condition and other regularity conditions, the probability of the Lasso selecting wrong models satisfies $P(\hat S \neq S)=o(e^{-n^{c_2}})$ \cite{ZhaoYu2006}. Then applying the above theorems to Lasso+mLS and Lasso+Ridge as special cases, we can easily derive their convergence rates of bias and MSE, see Section 3 for more details.

\subsection{Asymptotic Normality of Select+mLS and Select+Ridge}

In this section, we show asymptotic normality of Select+mLS and Select+Ridge. Let $\hat \Psi$ and $\Psi$ be the distribution functions of $\sqrt{n}( \tilde \beta_S - \beta^*_S )$ and $N(0,\sigma^2C_{11}^{-1})$ respectively, where $\tilde \beta$ can be any one of the two post model selection estimators: Select+mLS and Select+Ridge. Let $\hat S$ be the selected predictor set,

\begin{theorem}
\label{theorem: asymptotic normality}
Suppose that assumptions (a)-(e) are satisfied and that the model selection procedure is consistent, i.e., $P(\hat S \neq S)=o(1)$, then Select+mLS and Select+Ridge are asymptotically normal\footnote{For Select+Ridge, we need assumption (g) proposed in the next subsection to hold. Due to the restricted space, we don't state it separately.}, that is,
\begin{equation}
 \mathop {sup}\limits_{t \in R^s} | \hat \Psi(t) - \Psi(t)|  \rightarrow 0, \ as\ n\rightarrow \infty.
 \label{eqn:normality}
\end{equation}
\end{theorem}

This theorem states that model selection consistency in the first stage implies the asymptotic normality of the second stage estimators: Select+mLS and Select+Ridge. The proof of this theorem can be found in Appendix~\ref{C}.

\subsection{Residual Bootstrap after Select+mLS and Select+Ridge}

To make reliable scientific discoveries, we need to establish valid inference procedures including constructing confidence regions and testing for the parameter estimation. Although we have derived the asymptotic normality of Select+mLS and Select+Ridge, it is difficult to use in practice because the noise variance $\sigma^2$ is not known and hard to estimate in a high-dimensional setting. The bootstrap is a popular alternative in this case. A summary of bootstrap methods in linear and generalized linear penalized regression models for fixed $p$ can be found in \cite{Phdthesis}. We will consider the $p\gg n$ case by proposing a new inference procedure: residual bootstrap after two stage estimators. Our method allows $p$ to grow at an exponential rate in $n$.

In the context of a regression model, residual bootstrap is a standard method to bootstrap when the design matrix $X$ is deterministic \cite{Efron1979,Freedman1981,KnightFu2000}. Let $\tilde \beta$ denote Select+mLS or Select+Ridge, the residual vector is given by:
\begin{equation}
 \hat \epsilon = (\hat \epsilon_1,...,\hat \epsilon_n)^T = Y - X \tilde \beta.
 \label{eqn:residual}
\end{equation}
Consider the centered residuals at the mean $\{\hat \epsilon_i - \hat \mu,\ i=1,...,n\}$, where $\hat \mu = \frac{1}{n} \sum\limits_{i=1}^{n} {\hat \epsilon_i}$. For residual bootstrap, one obtains $\epsilon^*=(\epsilon_1^*,...,\epsilon_n^*)^T$ by resampling with replacement from the centered residuals $\{\hat \epsilon_i - \hat \mu,\ i=1,...,n\}$, and formulates $Y^*$ as follows
\begin{equation}
 Y^* = X \tilde \beta + \epsilon^*.
 \label{eqn:resampleY}
\end{equation}
Then one can define the selected predictor set $\hat S^*$ and Select+mLS or Select+Ridge $\tilde \beta^*$ based on the bootstrap sample $(X,Y^*)$. For the bootstrap to be valid, one needs to verify that the conditional distribution of $T_n^*= \sqrt{n}(\tilde \beta^*-\tilde \beta)$ given $\epsilon$, which can be computed directly from the data, approximates the distribution of $T_n = \sqrt{n}(\tilde \beta-\beta^*)$. The difference between two distributions can be characterized by Mallows metric.
\begin{definition}
The Mallows metric $d$, relative to the Euclidean norm $||\cdot||$, of two distributions $F$ and $\tilde F$ is the infimum of
$(E||Z-W||_2^2)^{\frac{1}{2}}$ over all pairs of random vectors $Z$ and $W$, where $Z$ has distribution $F$ and $W$ has distribution $\tilde F$. That is,
\begin{equation}
 d(F,\tilde F) = \mathop {inf}\limits_{Z \sim F,W \sim \tilde F} (E||Z-W||_2^2)^{\frac{1}{2}}.
 \label{eqn:mallow}
\end{equation}
\label{eqn:Mallows}
\end{definition}
By Lemma 8.1 in \cite{BickelFreedman1981a}, the infimum can be attained. To proceed, denote $G_n$ and $G_n^*$ the distribution of $T_n$ and the conditional distribution of $T_n^*$ respectively. Let $P^*$ denote the conditional probability given the error variables $\{\epsilon_i,i=1,...,n\}$.
To show the validity of residual bootstrap after Select+mLS or Select+Ridge, we need more conditions:
\ \\

{\em Assumption (dd): suppose that $s^2/n\rightarrow 0,\ as \ n\rightarrow \infty$.
}

{\em Assumption (f): suppose that both the probability of selecting wrong models based on the original data $(X,Y)$ and that based on the resample $(X,Y^*)$ decay at an exponential rate, i.e. $ P(\hat S \neq S) = o(e^{-n^{c_2}})$ and $ P^*(\hat S^* \neq \hat S) = o_p(e^{-n^{c_2}})$.
}

{\em Assumption (g): suppose that
\begin{equation}
 \frac{1}{n}||X\beta^*||_2^2 =O(n).
 \label{asum:f}
\end{equation}
}

{\em Assumption (h): suppose that $\mathop {max}\limits_{1\leq i \leq n}\sum\limits_{j=1}^{s}x_{ij}^2=o(n^{\frac{1}{2}})$.
}
\ \\

Assumption (dd) is stronger than assumption (d) because it requires that $s$ grows slower than $\sqrt{n}$. Without considering model selection, \cite{BickelFreedman1983} showed that residual bootstrap OLS fails if $p^2/n$ does not tend to 0. Therefore, the assumption (dd) cannot be weakened. As we shown in the next section, assumption (f) is satisfied if the Irrepresentable condition and some regularity conditions hold. Assumption (g) is a technical assumption which makes the convergence rates in Theorems~\ref{theorem: asymptotic bias and variance} and \ref{theorem: asymptotic bias and variance ridge} more clear. If we suppose $\Lambda_{max}(C_{11}) = \Lambda_{max}(\frac{1}{n}X_S^T X_S) \leq \Lambda_{max} <\infty$ where $\Lambda_{max}$ is a constant, assumption (g) is equivalent to $||\beta^*||_2^2 =O(n)$ because
$$\Lambda_{min} ||\beta^*||_2^2 \leq \frac{1}{n}||X\beta^*||_2^2 = \frac{1}{n}||X_S\beta^*_S||_2^2 \leq \Lambda_{max}(\frac{1}{n}X_S^T X_S)||\beta^*||_2^2 \leq \Lambda_{max} ||\beta^*||_2^2.$$
Since $\beta^*$ has only $s \ll n$ nonzero components, this assumption is not very restrictive. Obviously, it is satisfied when the maximum of $\beta^*_j$ is upper bounded by a constant. Assumption (h) is not very restrictive either because the number of terms in the sum is $s\ll n^{\frac{1}{2}}$ and it clearly holds when all the predictors corresponding to the nonzero coefficients are bounded by a constant $M$, i.e. $|x_{ij}|\leq M,\ i=1,...,n,\ j=1,...,s$. \cite{Huang2008} also assumed this condition to show the asymptotic normality of Bridge estimator.

Let ``$\rightarrow_p$'' denote convergence in probability,
\begin{theorem}
\label{theorem:bootstraplassovalid}
Suppose that assumptions (a)-(h) and (dd) are satisfied, then $d(G_n,G_n^*)$ converges in probability to zero, i.e.,
\begin{equation}
 d(G_n,G_n^*)\rightarrow_p 0.
 \label{eqn:bootstrap}
\end{equation}
\end{theorem}

Theorem~\ref{theorem:bootstraplassovalid} states that residual bootstrap after Select+mLS or Select+Ridge gives a valid approximation to the distribution of $T_n$ if the probabilities of selecting wrong models $P(\hat S \neq S) $ and $P^*(\hat S^* \neq \hat S)$ decay at an exponential rate and the number of true predictors $s$ grows slower than $\sqrt{n}$. The proof of this theorem can be found in Appendix~\ref{C}.

In the next section, we will apply the above Theorems~\ref{theorem: asymptotic bias and variance}, \ref{theorem: asymptotic bias and variance ridge}, \ref{theorem: asymptotic normality} and \ref{theorem:bootstraplassovalid} to two special cases: Lasso+mLS and Lasso+Ridge.

\section{Asymptotic Properties of Lasso+mLS and Lasso+Ridge}

As we shown in Section 2, Select+mLS and Select+Ridge have attractive asymptotic properties (see Theorem~\ref{theorem: asymptotic bias and variance}, \ref{theorem: asymptotic bias and variance ridge}, \ref{theorem: asymptotic normality} and \ref{theorem:bootstraplassovalid}) if the probability of selecting wrong models decays exponentially fast. Applying these theorems to Lasso+mLS and Lasso+Ridge, we can easily attain their convergence rates of bias and MSE, asymptotic normality and the validity of residual bootstrap. To proceed, we will first give a brief overview of the assumptions used to get model selection consistency of the Lasso investigated by \cite{ZhaoYu2006}. Let $sign(\cdot)$ map positive entries to $1$, negative entries to $-1$ and zero to zero,
 \begin{definition}[{\bf Irrepresentable condition} \cite{Fuchs2005,Tropp,MeinshausenBuhlmann2006,ZhaoYu2006,Wainwright2009}]
There exists a positive constant vector $\eta$, such that
\begin{equation}
 |C_{21}C_{11}^{-1}sign(\beta^*_S)|\leq \mathbf{1}-\eta
 \label{eqn:IC}
\end{equation}
where $\mathbf{1}$ is a $p-s$ by $1$ vector with entries $1$ and the inequality holds element-wise.
\end{definition}

{\em Assumption (i): there exist constant $c_1+c_2<c_3\leq 1$ and $M>0$ so that
\begin{equation}
 n^{\frac{1-c_3}{2}} \mathop {min}\limits_{1\leq i \leq s}|\beta_i^*| \geq M.
 \label{eqn:cond4}
\end{equation}
}

{\em Assumption (j): suppose that the tuning parameter $\lambda_n$ in the definition of the Lasso satisfies $\lambda_n \propto n^{\frac{1+c_4}{2}}$ with $c_2<c_4<c_3-c_1$.
}
\ \\

The Irrepresentable Condition is a key assumption on the design that can be weakened by e.g. using stability selection instead of the Lasso. \cite{MeinshausenBuhlmann2010} showed that stability selection can achieve the same convergence rate of the probability of selecting wrong models as the Lasso does but under a less restrictive sparse eigenvalue condition. We will give a brief overview of stability selection combined with randomized Lasso in sparse linear regression in the Appendix~\ref{B} and discuss the asymptotic properties of two stage estimators based on stability selection in the next section. Assumption (i) requires a gap of size $n^{\frac{c_3}{2}}$ between the decay rate of $\beta^*$ and $n^{-\frac{1}{2}}$ thus preventing the estimation from being dominated by the noise terms. It is weaker than $\mathop {min}\limits_{1\leq i \leq s}|\beta_i^*| \geq M$, which was assumed by \cite{Huang2008} who studied the asymptotic properties of the Bridge estimator. \cite{FanLv2008} also proposed a similar constraint $\mathop {min}\limits_{1\leq i \leq s}|\beta_i^*| \geq \frac{M}{n^\kappa},\ 0 \leq \kappa < \frac{1}{2}$ to show model selection consistency of the Sure Independent Screening.

Now, we are ready to state our main results. Let $\tilde \beta$ and $\tilde \beta^*$ denote the Lasso+mLS (or Lasso+Ridge) based on the original data $(X,Y)$ and that based on the resample $(X,Y^*)$ respectively, then we can define $\hat \Psi$, $\Psi$, $T_n$, $T_n^*$, $G_n$ and $G_n^*$ as Section 2 does.

Firstly, combining Theorem~\ref{theorem: asymptotic bias and variance} and the model selection property of the Lasso (see Lemma~\ref{lemma: model selection} in Appendix~\ref{A}), we can attain the following corollary:

\begin{corollary}
\label{corollary: asymptotic bias and variance}
Suppose that assumptions (a)-(d), (g), (i), (j) and the Irrepresentable condition $\eqref{eqn:IC}$ are satisfied, if $\tau_n \propto \frac{1}{n}$ and $\mu_n \propto e^{-n^{c_2}/4}$, then the bias and the MSE of Lasso+mLS and Lasso+Ridge estimators $\tilde \beta$ satisfy
\begin{equation}
|| E \tilde \beta - \beta^*||_2^2 = o(e^{-n^{c_2}/2})\rightarrow 0,\ as \ n\rightarrow \infty,
\label{eqn:corollarybias}
\end{equation}
\begin{equation}
E||\tilde \beta - \beta^*||_2^2 = O( \frac{\sigma^2}{\Lambda_{min}} \frac{s}{n}).
\label{eqn:corollarymse}
\end{equation}
\end{corollary}

\begin{proof} We only consider the Lasso+mLS since the proof for Lasso+Ridge is similar. By Lemma~\ref{lemma: model selection} in Appendix~\ref{A}, we have
\[ P(\hat S \neq S)\leq P(sign(\hat \beta(\lambda_n)) \neq sign( \beta^*)) \leq o(e^{-n^{c_2}}). \]
From assumption (c) \eqref{eqn:cond2}, we know that $\Lambda_{min} (C_{11})  \geq \Lambda_{min} >0$. And since $C_{11}$ is an $s$ by $s$ matrix, we have
\[ \frac{\sigma^2}{n} tr(C_{11}^{-1}) \leq \frac{\sigma^2s}{n}\Lambda_{min}^{-1}. \]
Under condition (g) or equation \eqref{asum:f}, we have
\[ ||\beta^*||_2^2 \leq \Lambda_{min}^{-1} \frac{1}{n}||X\beta^*||_2^2 = O(n). \]
The corollary is obtained directly from Theorem~\ref{theorem: asymptotic bias and variance}.
\end{proof}

Corollary~\ref{corollary: asymptotic bias and variance} indicates that Lasso+mLS and Lasso+Ridge are asymptotically unbiased. In particular, their biases decay at an exponential rate and their MSEs achieve the oracle convergence rate of $\frac{s}{n}$ which is much faster than that of the Lasso. (Under the restricted eigenvalue condition, the Lasso achieves convergence rate of $\frac{s \log p}{n}$ for MSE, see for example \cite{Raskutti2011,ZhangHuang2008}).

Secondly, combining Theorem~\ref{theorem: asymptotic normality} and Lemma~\ref{lemma: model selection}, we can easily derive the asymptotic normality of Lasso+mLS and Lasso+Ridge.

\begin{corollary}
\label{corollary: asymptotic normality}
Suppose conditions (a)-(e), (i), (j) and the Irrepresentable Condition $\eqref{eqn:IC}$ are satisfied, then Lasso+mLS and Lasso+Ridge are asymptotically normal\footnote{For Lasso+Ridge, we need assumption (g) to hold. Due to the restricted space, we don't state it separately.}. That is,
\begin{equation}
 \mathop {sup}\limits_{t \in R^s} | \hat \Psi(t) - \Psi(t)|  \rightarrow 0, \ as\ n\rightarrow \infty.
 \label{eqn:normality}
\end{equation}
\end{corollary}
\ \\
{\bf Remark 3.1.} For fixed $p$ and fixed $\beta^*$, \cite{KnightFu2000} showed that for $\lambda_n = o(\sqrt{n})$, the Lasso estimator is also asymptotically normal $N(0,\sigma^2C^{-1})$ under conditions $\frac{1}{n} X^TX \rightarrow C$ and $\frac{1}{n} \mathop {max}\limits_{1\leq i \leq n} x_i^Tx_i \rightarrow 0$. Then the asymptotic covariance matrix of the rescaled and centered Lasso estimator $\sqrt{n}( \hat \beta_S - \beta^*_S )$ is $\sigma^2$  multiplied by
$$C_{11}^{-1} + C_{11}^{-1} C_{12}(C_{22}-C_{21}C_{11}^{-1}C_{12})C_{21}C_{11}^{-1}\ \  (\succeq C_{11}^{-1})$$
where matrix $A\succeq B$ means $A-B$ is positive definite. Therefore, Lasso+mLS and Lasso+Ridge have smaller asymptotic covariance matrix (which is $\sigma^2C_{11}^{-1}$) compared with the Lasso and hence reduce estimation uncertainty. Moreover, as pointed out by \cite{FanLi2001}, one cannot find a $\lambda_n$ such that the Lasso estimator is model selection consistent and asymptotically normal ($\sqrt{n}-$consistency) simultaneously. In this sense, Lasso+mLS and Lasso+Ridge improve the performance of the Lasso.

Lastly, we verify the validity of residual bootstrap after Lasso+mLS and Lasso+Ridge. We would like to begin with showing an interesting result that the Lasso estimator $\hat \beta^*$ based on the resample $(X,Y^*)$ also has model selection consistency.
\begin{equation}
 \hat \beta^* = \mathop {argmin}\limits_{\beta} \left\{ || {Y^* - X \tilde \beta } ||^2  + \lambda_n || {\beta } ||_1 \right\}.
 \label{eqn:bootstraplasso}
\end{equation}
Recall that $\hat S= \{ j \in \{1,2,...,p\}:\ \tilde \beta_j \neq 0 \}$ and $\hat S^* = \{ j \in \{1,2,...,p\}:\ \hat \beta_j^* \neq 0 \}$ are the sets of selected predictors by $\tilde \beta$ and $\hat \beta^*$ respectively.

\begin{lemma}
\label{lemma:bootstrap selection consistency}
Suppose conditions (a)-(e), (g)-(j), (dd) and the Irrepresentable Condition $\eqref{eqn:IC}$ are satisfied, then the following holds,
\[ P^*(\hat S^* \neq \hat S) = o_p(e^{-n^{c_2}}). \]
\end{lemma}

Now, applying Theorem~\ref{theorem:bootstraplassovalid}, Lemma~\ref{lemma:bootstrap selection consistency} and Lemma~\ref{lemma: model selection}, we can state our main result:
\begin{corollary}
\label{corollary:bootstraplassovalid}
Suppose that conditions (a)-(e), (g)-(j), (dd) and the Irrepresentable Condition $\eqref{eqn:IC}$ are satisfied, then residual bootstrap after Lasso+mLS or Lasso+Ridge is consistent in the sense that
\begin{equation}
 d(G_n,G_n^*)\rightarrow_p 0.
 \label{eqn:bootstrap}
\end{equation}
\end{corollary}
\ \\
{\bf Remark 3.2.} In practice, if the Irrepresentable Condition does not hold, one needs to try other model selection methods, e.g., Bolasso \cite{Bach2008} or stability selection. As stated before, as long as the probabilities of selecting wrong models $P(\hat S \neq S) $ and $P^*(\hat S^* \neq \hat S)$ decay at an exponential rate $o(e^{-n^{c_2}})$, residual bootstrap after two stage estimator gives valid approximation.

Corollary~\ref{corollary:bootstraplassovalid} indicates that residual bootstrap after Lasso+mLS or Lasso+Ridge gives a valid approximation to the distribution of $T_n$ and then can be used to construct confidence intervals and test for parameter estimation. The proof is straightforward and we omit it.

\section{Asymptotic Properties of Modified Least Squares and Ridge after Stability Selection}

If the Irrepresentable Condition is violated, the Lasso cannot correctly select the true model. In this case, one needs to apply other model selection criteria instead of the Lasso. Stability selection is one popular method among many others. \cite{MeinshausenBuhlmann2010} showed that it can achieve the same convergence rate of the probability of selecting wrong models but under a less restrictive sparse eigenvalue condition. More details are given in the Appendix~\ref{B}.

Let $\tilde \beta$ and $\tilde \beta^*$ denote the modified Least Squares (or Ridge) after stability selection (SS+mLS or SS+Ridge) based on the original data $(X,Y)$ and that based on the resample $(X,Y^*)$ respectively, then we can define $\hat \Psi$, $\Psi$, $T_n$, $T_n^*$, $G_n$ and $G_n^*$ as Section 2 does. Applying the asymptotic properties of Select+mLS and Select+Ridge in Section 2, we can derive without any proofs the following corollaries parallel to Corollary~\ref{corollary: asymptotic bias and variance}, \ref{corollary: asymptotic normality}, \ref{corollary:bootstraplassovalid}.

\begin{corollary}
\label{corollary: asymptotic bias and variance for stability selection}
Assume the conditions in Lemma~\ref{lemma: model selection stability} in the Appendix~\ref{B} and conditions (b)-(c) and (g) in the previous sections are satisfied, if $\tau_n \propto \frac{1}{n}$ and $\mu_n \propto e^{-n^{c_2}/4}$, then the bias and the MSE of the SS+mLS and SS+Ridge $\tilde \beta$ satisfy
\begin{equation}
|| E \tilde \beta - \beta^*||_2^2 = o(e^{-n^{c_2}/2})\rightarrow 0,\ n\rightarrow \infty,
\label{eqn:corollarybias}
\end{equation}
\begin{equation}
E||\tilde \beta - \beta^*||_2^2 = O( \frac{\sigma^2}{\Lambda_{min}} \frac{s}{n}).
\label{eqn:corollarymse}
\end{equation}
\end{corollary}

\begin{corollary}
\label{corollary: asymptotic normality for stability selection}
Assume the conditions in Lemma~\ref{lemma: model selection stability} in the Appendix~\ref{B} and conditions (b)-(e), (g)-(h) and (dd) in the previous sections are satisfied, then the SS+mLS and SS+Ridge $\tilde \beta$ are asymptotically normal\footnote{For SS+Ridge, we need assumption (g) to hold. Due to the restricted space, we don't state it separately.}, that is,
\begin{equation}
 \mathop {sup}\limits_{t \in R^s} | \hat \Psi(t) - \Psi(t)|  \rightarrow 0, \ as\ n\rightarrow \infty.
 \label{eqn:normality}
\end{equation}
\end{corollary}

\begin{corollary}
\label{corollary:bootstraplassovalid for stability selection}
Assume the conditions in Lemma~\ref{lemma: model selection stability} in the Appendix~\ref{B} and conditions (b)-(c), (e), (g)-(h) and (dd) in the previous sections are satisfied, then residual bootstrap after SS+mLS or SS+Ridge is consistent in the sense that
\begin{equation}
 d(G_n,G_n^*)\rightarrow_p 0.
 \label{eqn:bootstrap}
\end{equation}
\end{corollary}

\section{Simulation}

In this section we carry out simulation studies to evaluate the finite sample performance of Lasso+mLS and Lasso+Ridge. We have also constructed simulations for SS+mLS and SS+Ridge (modified Least Squares or Ridge after stability selection). Though in theory stability selection achieves the same convergence rate of the probability of selecting wrong models under weaker conditions compared with the Lasso, their finite sample performance is similar to the Lasso unless the signal to noise ratio is very high. In our simulation, Lasso+mLS and Lasso+Ridge work well and perform similarly with SS+mLS and SS+Ridge regardless of whether the Irrepresentable condition holds or not. Therefore we only present here the results for Lasso+mLS and Lasso+Ridge.

In the following simulations, we also compared the performance of the Lasso+mLS (with $\tau_n=1/n$) with that of the Lasso+OLS and found that their finite sample results are almost the same. This is true because the smallest singular value of the matrix $X_{\hat S}$ containing the predictors selected by Lasso is bounded well from $0$ in all examples which makes the hard thresholding step not necessary. We will omit the results of Lasso+OLS for the sake of brevity.

\subsection{Comparison of Bias$^2$, MSE and PMSE}

This subsection compares the bias$^2$ ($||E \tilde \beta - \beta^*||_2^2$), MSE and prediction mean squared error (PMSE) of the Lasso+mLS and Lasso+Ridge with those of the Lasso. We use R package ``glmnet" to compute the Lasso solution. As part of the simulation, we fix $\tau_n=1/n$ and $\mu_n=1/n$, so the only tuning parameter for all the three methods is $\lambda_n$ which can be chosen by a 5-fold cross-validation.

Our simulated data are drawn from the following model
\[ y=x^T \beta^* + \epsilon,\ \epsilon \sim N(0,\sigma^2). \]
We set $p=500$, $s=10$ and $\sigma=1$. The predictor vector x is generated from a multivariate normal distribution $N(0,\Sigma)$. The value of x is generated once and then kept fixed. We consider two different Toeplitz covariance matrices $\Sigma$ which control the correlation among the predictors: (1) $\Sigma=I$ and (2) $\Sigma_{ij}=\rho^{|i-j|}$ where $\rho=0.5$. For the true parameter $\beta^*$, the first $s=10$ elements are nonzero with two different patterns of sign:
\begin{flushleft}
$case\ (1):\ \  \beta^*_{1-10}=\{1.5,1.5,1.5,1.5,1.5,0.75,0.75,0.75,0.75,0.75\},$

$case\ (2):\ \ \beta^*_{1-10}=\{1.5,1.5,-1.5,-1.5,1.5,0.75,-0.75,0.75,-0.75,-0.75\}. $
\end{flushleft}
The remaining $p-s=490$ elements of $\beta^*$ are zero. Table~\ref{tab:examplesettings} summaries eight different example settings.

\begin{table}[ht]
\begin{center}
 \begin{threeparttable}
 \caption{\label{tab:examplesettings} Example settings.}
\begin{tabular}{cccccc}
  \hline
example & n & p & s & $\Sigma_{ij},i\neq j$ & $\beta^*$ \\
  \hline
1 & 200 & 500 & 10 & 0 & case (1) \\   \hline
 2 & 400 & 500 & 10 &  0  & case (1) \\   \hline
 3 & 200 & 500 & 10 & $0.5^{|i-j|}$ & case (1) \\   \hline
4 & 400 & 500 & 10 &  $0.5^{|i-j|}$ & case (1) \\   \hline
5 & 200 & 500 & 10 & 0 & case (2) \\   \hline
 6 & 400 & 500 & 10 &  0  & case (2) \\   \hline
 7 & 200 & 500 & 10 & $0.5^{|i-j|}$ & case (2) \\   \hline
8 & 400 & 500 & 10 &  $0.5^{|i-j|}$ & case (2) \\   \hline
\end{tabular}
 \end{threeparttable}
\end{center}
\end{table}

After $X$ was generated, we examined the Irrepresentable condition and found that it holds in examples $1-6$ and is violated in examples $7-8$. In order to evaluate the prediction performance, we generate an independent testing data set of size $500$ and compute the PMSE. Summary statistics are calculated based on $100$ replications (keeping $X$ fixed) and showed in Table~\ref{tab:example1} and Figure~\ref{fig:biasmsepmse}.

We see that Lasso+mLS and Lasso+Ridge perform almost the same. They not only dramatically decrease the bias$^2$ of the Lasso by more than $90\%$ but also can reduce the variance (in fact, their variances are $20\%-55\%$ smaller than that of the Lasso in all examples except in example 7 where their variances are $45\%$ larger), therefore they improve the MSE and PMSE by $40\%-80\%$ and $5\%-25\%$ respectively. These benefits occur regardless of whether the Irrepresentable condition holds or not, which indicates that Lasso+mLS and Lasso+Ridge dominate the performance of the Lasso in terms of estimation.

\begin{table}[ht]
 \centering
 \begin{threeparttable}
 \caption{\label{tab:example1}Comparison of Lasso, Lasso+mLS and Lasso+Ridge in terms of bias$^2$, MSE and PMSE.}
\begin{tabular}{rllll}
  \hline
Example &  & Lasso & Lasso+mLS & Lasso+Ridge \\
  \hline
1 &bias$^2$ & 0.266 & \bf{0.003} & 0.005 \\
   & MSE & 0.42(0.11) & 0.08(0.05) & \bf{0.07(0.04)} \\
   & PMSE & 1.45(0.15) & \bf{1.08(0.08)} & 1.08(0.08) \\  \hline
  2 &bias$^2$ & 0.081 & \bf{0.001} & 0.001 \\
   & MSE & 0.13(0.04) & 0.03(0.03) & \bf{0.03(0.01)} \\
   & PMSE & 1.14(0.08) & \bf{1.04(0.07)} & 1.04(0.07) \\  \hline
  3 &bias$^2$ & 0.062 & \bf{0.001} & 0.001 \\
   & MSE & 0.18(0.06) & \bf{0.09(0.05)} & 0.09(0.05) \\
   & PMSE & 1.24(0.11) & 1.07(0.08) & \bf{1.07(0.07)} \\  \hline
 4  &bias$^2$ & 0.02 & \bf{0.001} & 0.001 \\
   & MSE & 0.08(0.03) & \bf{0.04(0.02)} & 0.04(0.02) \\
   & PMSE & 1.09(0.07) & \bf{1.04(0.07)} & 1.04(0.07) \\  \hline
  5 &bias$^2$ & 0.212 & \bf{0.001} & 0.002 \\
   & MSE & 0.35(0.09) & 0.06(0.04) & \bf{0.06(0.03)} \\
   & PMSE & 1.36(0.14) & 1.08(0.08) & \bf{1.08(0.07)} \\  \hline
  6 &bias$^2$ & 0.097 & 0.0002 & \bf{0.0001} \\
   & MSE & 0.15(0.04) & 0.03(0.02) & \bf{0.03(0.01)} \\
   & PMSE & 1.16(0.08) & \bf{1.03(0.06)} & 1.03(0.06) \\  \hline
 7  &bias$^2$ & 0.494 & \bf{0.053} & 0.079 \\
   & MSE & 0.72(0.19) & \bf{0.39(0.31)} & 0.44(0.38) \\
   & PMSE & 1.48(0.16) & \bf{1.3(0.18)} & 1.32(0.2) \\  \hline
 8  &bias$^2$ & 0.343 & \bf{0.004} & 0.007 \\
   & MSE & 0.45(0.11) & \bf{0.07(0.04)} & 0.08(0.08) \\
   & PMSE & 1.26(0.09) & \bf{1.05(0.08)} & 1.05(0.08) \\
   \hline
  \end{tabular}
  \small
  $*$ The numbers in parentheses are the corresponding standard deviations.
 \end{threeparttable}
\end{table}

\begin{figure}[htbp]
\centering\includegraphics[width=\textwidth]{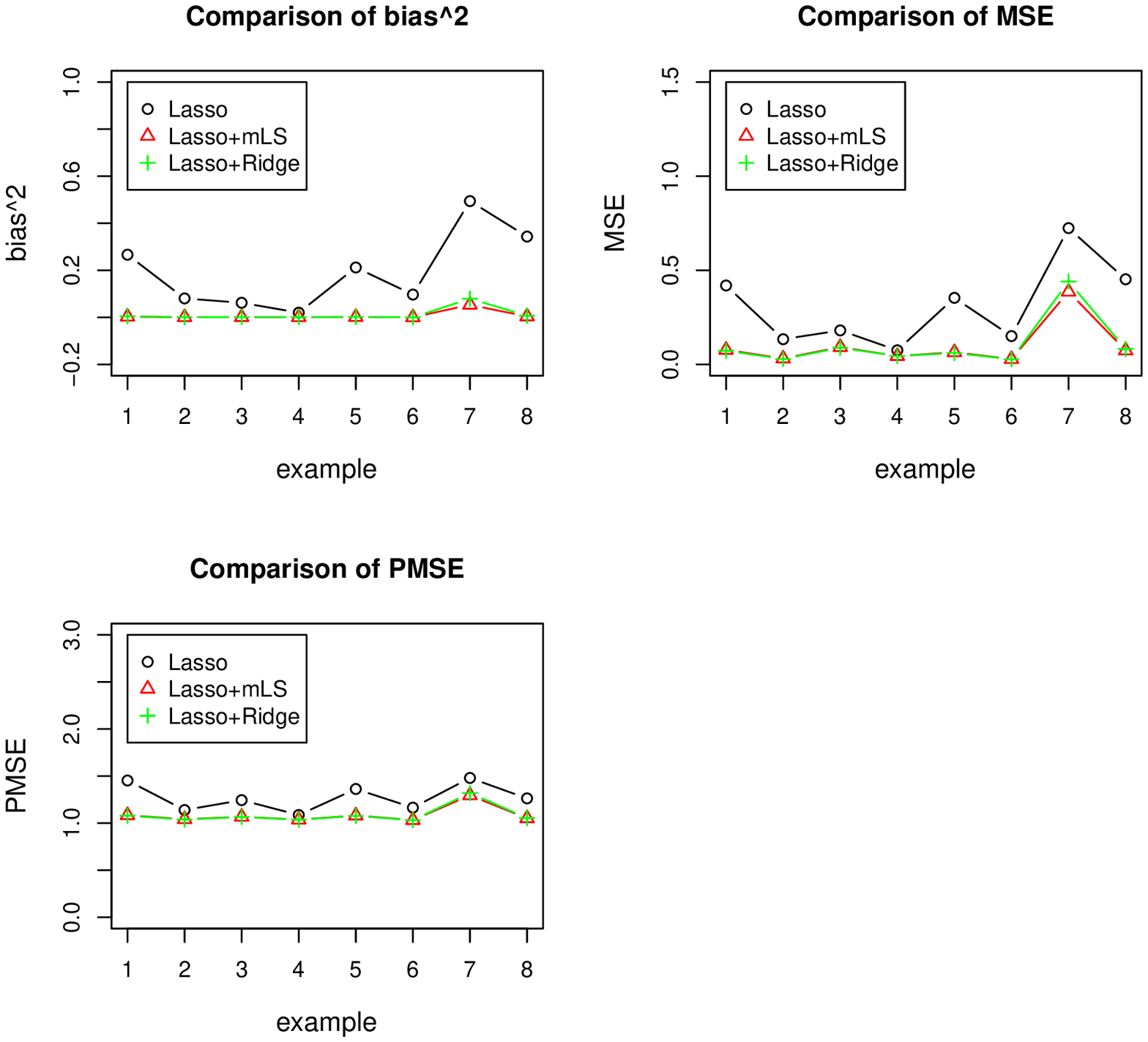}
\caption{Comparisons of bias$^2$, MSE and PMSE. Three methods are considered: Lasso (black circle), Lasso+mLS (red triangle) and Lasso+Ridge (green ``+"). Lasso+mLS and Lasso+Ridge behave similarly, both dominate the performance of the Lasso.}
\label{fig:biasmsepmse}
\end{figure}

\begin{figure}[htbp]
\centering\includegraphics[width=5in]{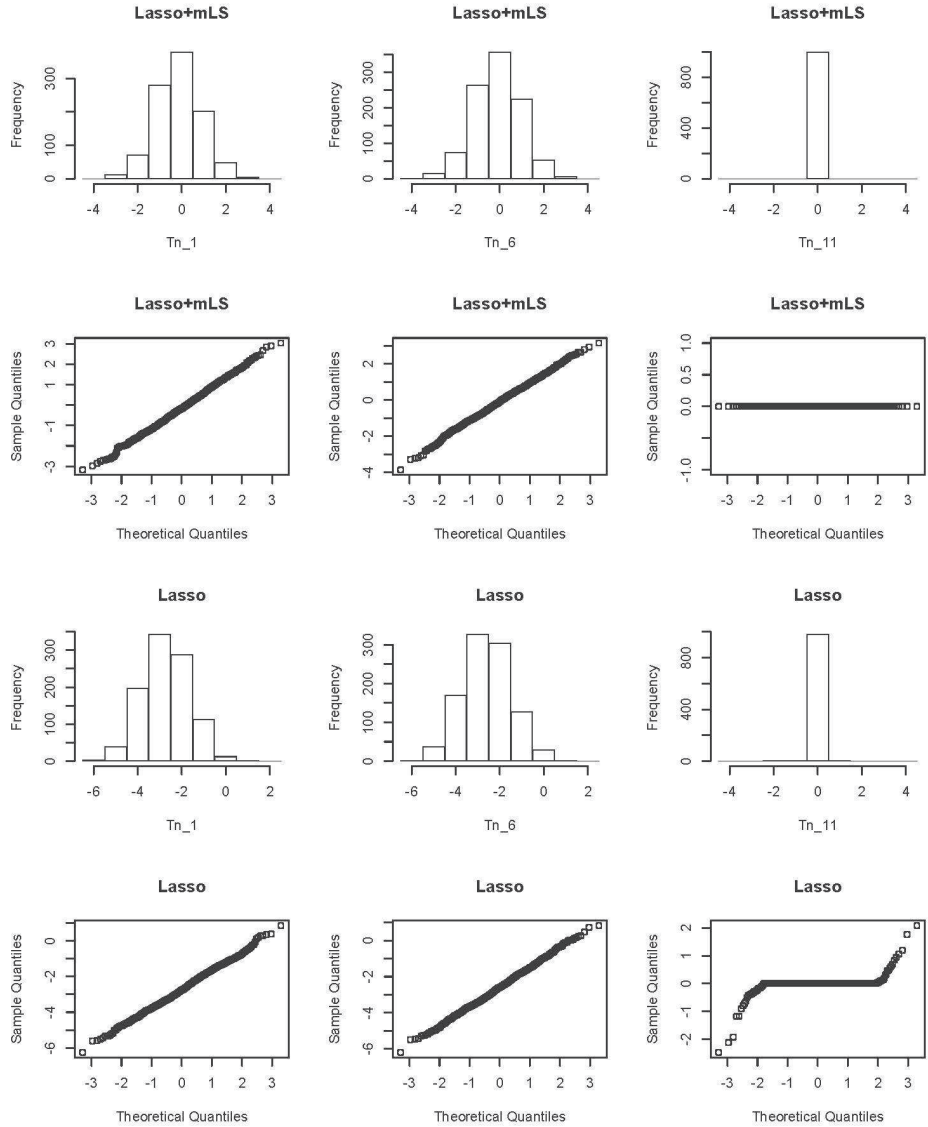}
\caption{Histograms and Normal Q-Q Plots of the scaled and centered Lasso and Lasso+mLS estimators in example 1 for $T_{n,j}=\sqrt{n}(\hat \beta_j -\beta_j^*), j=1,6,11$.}
\label{fig:hist_qqplot}
\end{figure}


\subsection{Finite Sample Distribution}
This section evaluates the finite sample distribution of the scaled and centered Lasso+mLS estimator $T_n=\sqrt{n}(\hat \beta-\beta^*)$. Lasso+Ridge behaves similarly, so we omit it.

We show in Figure~\ref{fig:hist_qqplot} the histograms and Normal Q-Q Plots of the scaled and centered Lasso and Lasso+mLS estimators based on $1000$ replications. We only present the results for individual coefficients $T_{n,j}$ in example $1$ with $j=1,6,11$ corresponding to the largest, the medium sized and the zero-valued coefficients respectively. Other coefficients in example $1$ and the coefficients in examples $2-8$ behave similarly (except example $7$ where the Irrepresentable condition does not hold). We can see that the finite sample distribution of the Lasso+mLS highly coincides with the asymptotically normal distribution which verifies the claims in Theorem $3$ and Corollary $2$. Although the finite sample distributions of the Lasso estimator for the largest and the medium sized coefficients also seem to somewhat resemble normality, the centers shift away from $0$.

We should mention that Lasso+mLS suffers the same issue proposed in \cite{Potscher} which studied the distribution of the adaptive Lasso estimator, that is, the finite sample distribution can be highly non-normal when there is not a gap between the decay rate of the nonzero $\beta^*$ and the order $n^{-1/2}$.

\subsection{Confidence Intervals and Coverage Probabilities}

In this subsection, we study the finite sample performance of residual bootstrap Lasso+mLS and Lasso+Ridge using the examples from Table~\ref{tab:examplesettings}. Since Lasso+mLS and Lasso+Ridge behave similarly, we only show the results of Lasso+mLS for the sake of brevity.

For each data set $(X,Y)$, we generated $500$ bootstrap samples $(X,Y^*)$ by residual bootstrap and then computed the Lasso and Lasso+mLS based on each bootstrap sample $(X,Y^*)$, both are denoted by $\hat \beta^*_{(b)},b=1,...,500$. Let $\hat t_{\alpha/2}$ and $\hat t_{1-\alpha/2}$ be the $\alpha/2$ and $1-\alpha/2$ quantiles of the empirical distribution of $\hat \beta^*$. Two approaches are considered to construct $1-\alpha$ confidence intervals for each individual parameter $\beta^*_j,j=1,...p$: (1) percentile confidence intervals defined by $[\hat t_{\alpha/2},\hat t_{1-\alpha/2}]$; and (2) basic confidence intervals defined by $[2\hat \beta- \hat t_{1-\alpha/2}, 2\hat \beta- \hat
t_{\alpha/2}]$ where $\hat \beta$ is the Lasso estimator for residual bootstrap Lasso or Lasso+mLS for residual bootstrap Lasso+mLS, see \cite{EfronTibshirani1993,DavisonHinkley1997}. This procedure is repeated $100$ times and then an estimate of the coverage probability is obtained.

We found that basic confidence intervals based on residual bootstrap Lasso provide more accurate coverage probabilities while having the same length as percentile confidence intervals (the basic $90\%$ confidence intervals can achieve coverage probabilities larger than $80\%$ while the percentile confidence intervals are too biased that their coverage probabilities can be lower than $20\%$ for the nonzero-valued parameters, see Figure~\ref{fig:basicpercentile}). The distribution of residual bootstrap Lasso is far away from being centered at the true value (Figure~\ref{fig:histogram}) which makes the percentile confidence intervals fail. Similar phenomenon happens for paired bootstrap Lasso method. Therefore, we suggest using the basic confidence intervals in practice with high-dimensional data. In what follows, our confidence intervals are all basic.

\begin{figure}[htbp]
\centering\includegraphics[width=\textwidth]{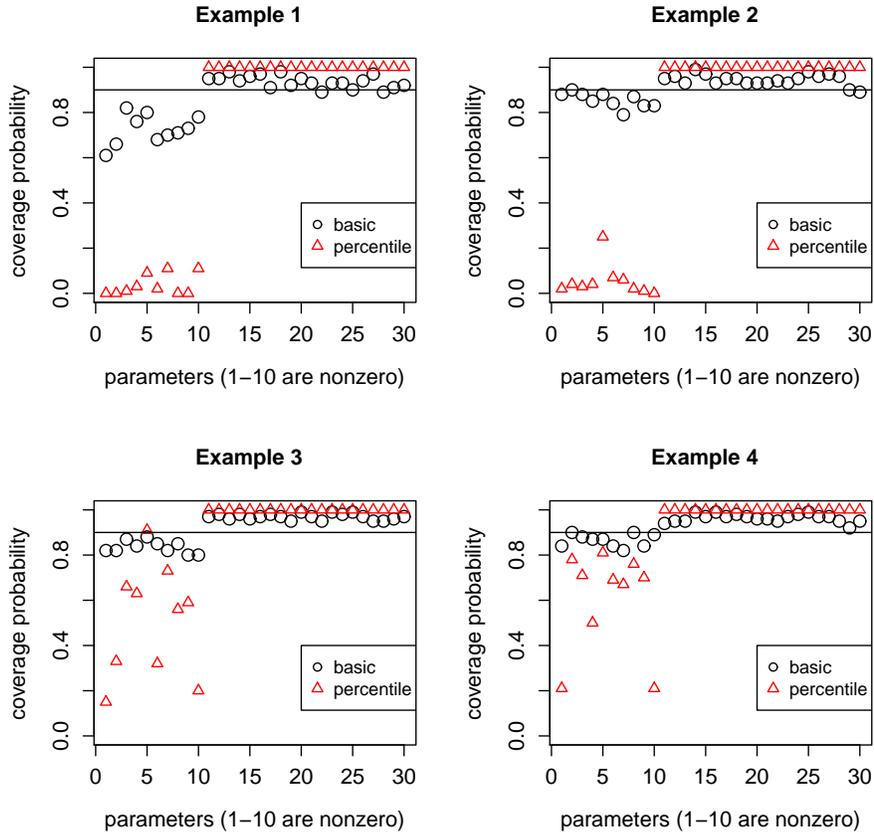}
\caption{Comparison of two confidence intervals construction approaches: basic (black circle) and percentile (red triangle). Coverage probabilities of $90\%$ confidence intervals for each $\beta^*_j,j=1,..,s,s+1,...,s+20$ based on residual bootstrap Lasso are shown in this figure. We only show the results for examples 1-4 since the results for examples 5-8 are similar. For a better view, the coverage probabilities for only $20$ zero-valued parameters are present and those for the remaining $p-s-20$ zero-valued parameters are similar to the $20$ presented and therefore are omitted. Basic confidence intervals provide much more accurate coverage probabilities than percentile confidence intervals.}
\label{fig:basicpercentile}
\end{figure}

\begin{figure}[htbp]
\centering\includegraphics[width=\textwidth]{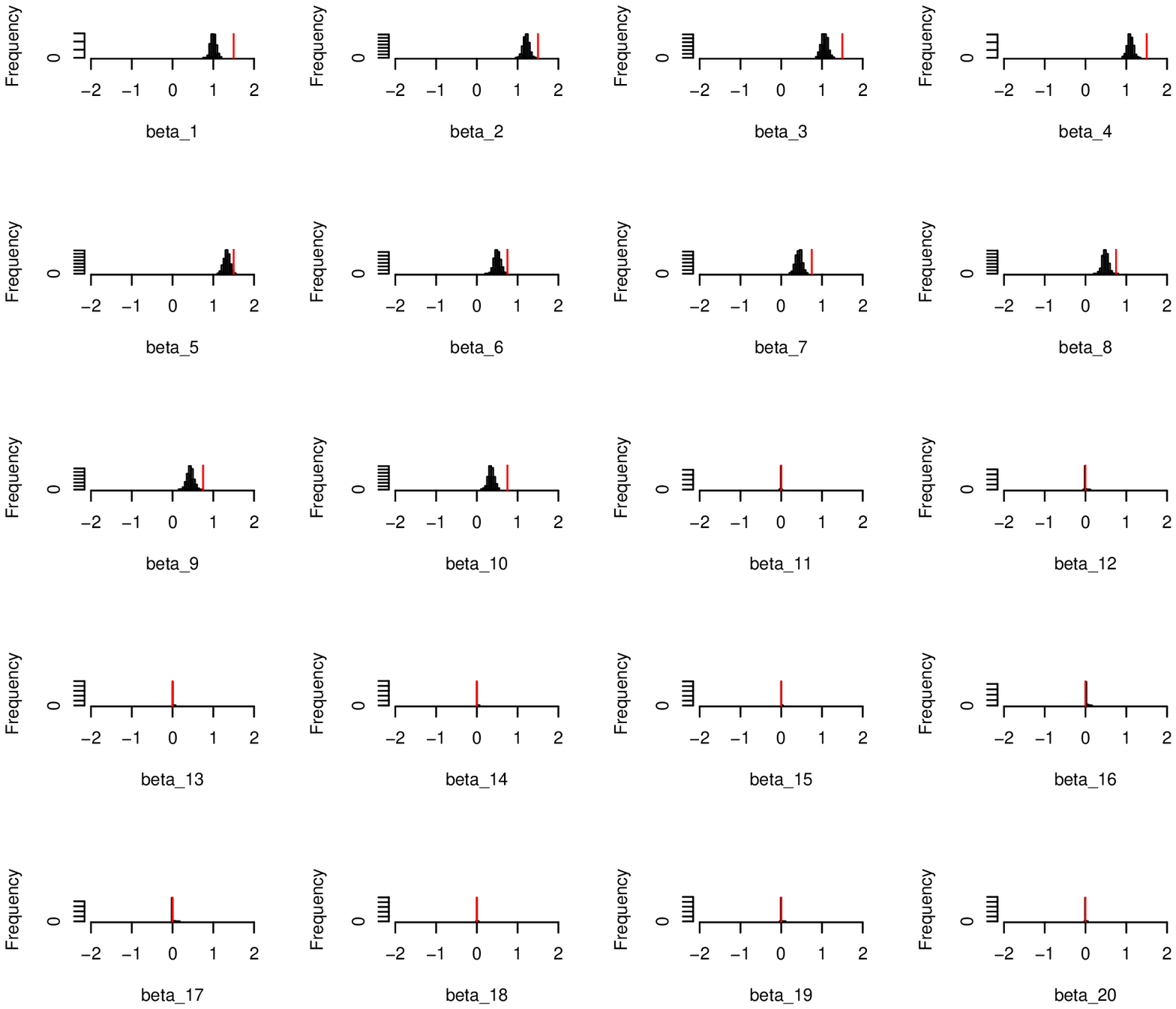}
\caption{Histograms of the distribution of residual bootstrap Lasso with the true value plotted as a red vertical line. This figure only show the results for the first $20$ parameters in example 1. Other parameters and other examples behave similarly. The large bias of the residual bootstrap Lasso makes percentile confidence intervals fail.}
\label{fig:histogram}
\end{figure}

Figure~\ref{fig:coverprobres} shows the coverage probabilities of $90\%$ confidence intervals based on residual bootstrap Lasso+mLS and residual bootstrap Lasso. In this figure, only $20$ zero-valued parameters are presented for the sake of brevity. The coverage probabilities for the remaining $p-s-20$ zero-valued parameters are similar to the $20$ presented and therefore are omitted. In addition, we average the coverage probabilities and interval lengths over nonzero-valued parameters part and zero-valued parameters part respectively (see Tables~\ref{tab:meancp} and \ref{tab:meanil}). From Figure~\ref{fig:coverprobres} and Tables~\ref{tab:meancp} and \ref{tab:meanil}, we can see that residual bootstrap Lasso+mLS gives accurate coverage probabilities (approximately $88\%$ for nonzero $\beta_j^*$ and $1$ for zero $\beta_j^*$) when the Irrepresentable condition holds (see examples 1-6), which verifies Corollary~\ref{corollary:bootstraplassovalid}. Note that, for zero-valued parameters, residual bootstrap Lasso+mLS produces confidence intervals with coverage probabilities close to $1$ and very short lengths (approximately $0$, see Figure~\ref{fig:length} and Table~\ref{tab:meanil}), reflecting the oracle properties in Corollary~\ref{corollary: asymptotic normality}. By contrast, residual bootstrap Lasso cannot provide accurate coverage probabilities unless $n$ is large enough (the coverage probability for nonzero $\beta^*_j$ is around $75\%$ for $n=200$ and is around $85\%$ for $n=400$). Even when the Irrepresentable condition doesn't hold (see examples 7-8), residual bootstrap Lasso+mLS can also provide reasonable coverage probabilities ($83.5\%$, $90.3\%$ for nonzero $\beta_j^*$, and $99.1\%$, $99.7\%$ for zero $\beta_j^*$) while residual bootstrap Lasso does not ($71.7\%$, $82.1\%$ for nonzero $\beta_j^*$ and $93.1\%$, $92.6\%$ for zero $\beta_j^*$). Compared with residual bootstrap Lasso, the coverage probability of residual bootstrap Lasso+mLS is about $7\%$ in average closer to the preassigned level $(90\%)$ for nonzero $\beta_j^*$ and $5\%$ closer to $1$ for zero $\beta_j^*$. Even though residual bootstrap Lasso has shorter ($17\%$ in average) interval lengths for the nonzero $\beta^*_j$, it loses accuracy in coverage. Overall, residual bootstrap Lasso+mLS is better than residual bootstrap Lasso. Moreover, when $n$ increases, the performance of both methods become better.

\begin{table}[ht]
\begin{center}
  \begin{threeparttable}
 \caption{\label{tab:meancp} Mean coverage probability of residual bootstrap Lasso (RBL), residual bootstrap Lasso+mLS (RBLmLS) and paired bootstrap Lasso (PBL). }
\begin{tabular}{rrrrrrrrrr}
  \hline
 Example & & 1 & 2 & 3 & 4 & 5 & 6 & 7 & 8 \\
  \hline
nonzero $\beta^*_j$ & RBL     & 0.725 & 0.855 & 0.835 & 0.865 & 0.792 & 0.875 & 0.717 & 0.821 \\
                                 & RBLmLS  & 0.880 & 0.882 & 0.880 & 0.884 & 0.901 & 0.917 & 0.835 & 0.903 \\
                                 & PBL      & 0.772 & 0.634 & 0.833 & 0.816 & 0.801 & 0.609 & 0.866 & 0.512 \\ \hline
  zero $\beta^*_j$     & RBL     & 0.937 & 0.947 & 0.965 & 0.968 & 0.940 & 0.947 & 0.931 & 0.926 \\
                                 & RBLmLS  & 0.999 & 1.000 & 1.000 & 1.000 & 1.000 & 1.000 & 0.991 & 0.997 \\
                                & PBL       & 0.990 & 0.996 & 0.997 & 0.999 & 0.991 & 0.997 & 0.987 & 0.992 \\
   \hline
\end{tabular}
 \end{threeparttable}
\end{center}
\end{table}

\begin{table}[ht]
\begin{center}
 \begin{threeparttable}
 \caption{\label{tab:meanil} Mean interval length of residual bootstrap Lasso (RBL), residual bootstrap Lasso+mLS (RBLmLS) and paired bootstrap Lasso (PBL). }
\begin{tabular}{rrrrrrrrrr}
  \hline
 Example & & 1 & 2 & 3 & 4 & 5 & 6 & 7 & 8 \\
  \hline
nonzero $\beta^*_j$ & RBL     & 0.209 & 0.158 & 0.278 & 0.202 & 0.222 & 0.165 & 0.266 & 0.203 \\
                                 & RBLmLS  & 0.248 & 0.178 & 0.322 & 0.241 & 0.254 & 0.204 & 0.350 & 0.290 \\
                                 & PBL     & 0.288 & 0.170 & 0.307 & 0.207 & 0.293 & 0.176 & 0.412 & 0.238 \\  \hline
  zero $\beta^*_j$     & RBL     & 0.011 & 0.004 & 0.002 & 0.001 & 0.009 & 0.004 & 0.017 & 0.014 \\
                                 & RBLmLS  & 0.001 & 0.000 & 0.000 & 0.000 & 0.000 & 0.000 & 0.009 & 0.002 \\
                                 & PBL      & 0.024 & 0.016 & 0.012 & 0.010 & 0.022 & 0.016 & 0.031 & 0.025 \\
   \hline
\end{tabular}
 \end{threeparttable}
\end{center}
\end{table}

\begin{figure}[htbp]
\centering\includegraphics[width=5in]{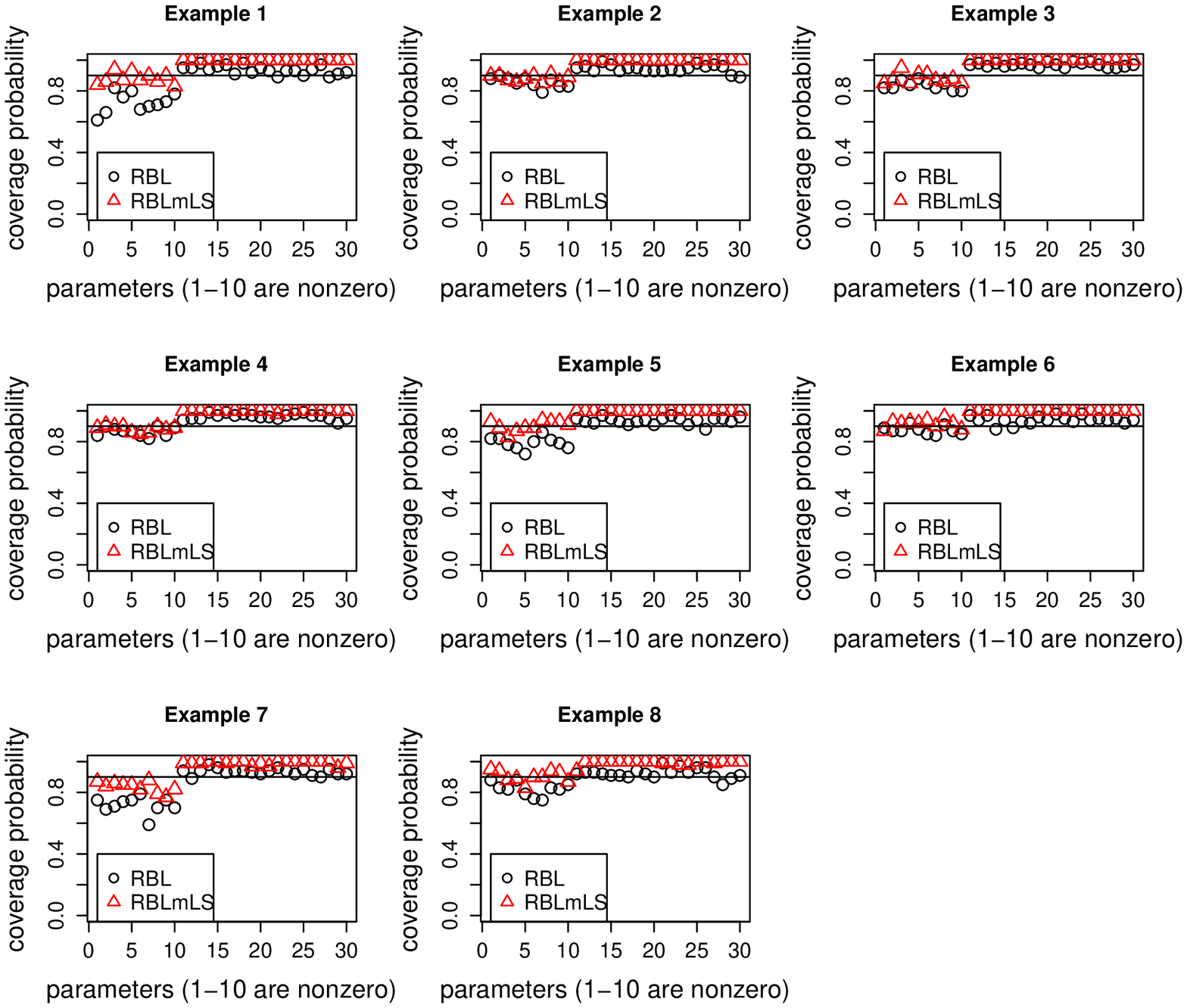}
\caption{Coverage probabilities of $90\%$ confidence intervals for each $\beta^*_j,j=1,..,s,s+1,...,s+20$ based on two methods: residual bootstrap Lasso+mLS (RBLmLS, red triangle) and residual bootstrap Lasso (RBL, black circle). For a better view, only $20$ zero-valued parameters are present. The results for the remaining $p-s-20$ zero-valued parameters are similar to the $20$ presented and therefore are omitted. Residual bootstrap Lasso+mLS provides more accurate coverage probabilities ($7\%$ in average closer to the preassigned level $(90\%)$ for nonzero $\beta_j^*$ and $5\%$ closer to $1$ for zero $\beta_j^*$).}
\label{fig:coverprobres}
\end{figure}

\begin{figure}[htbp]
\centering\includegraphics[width=5in]{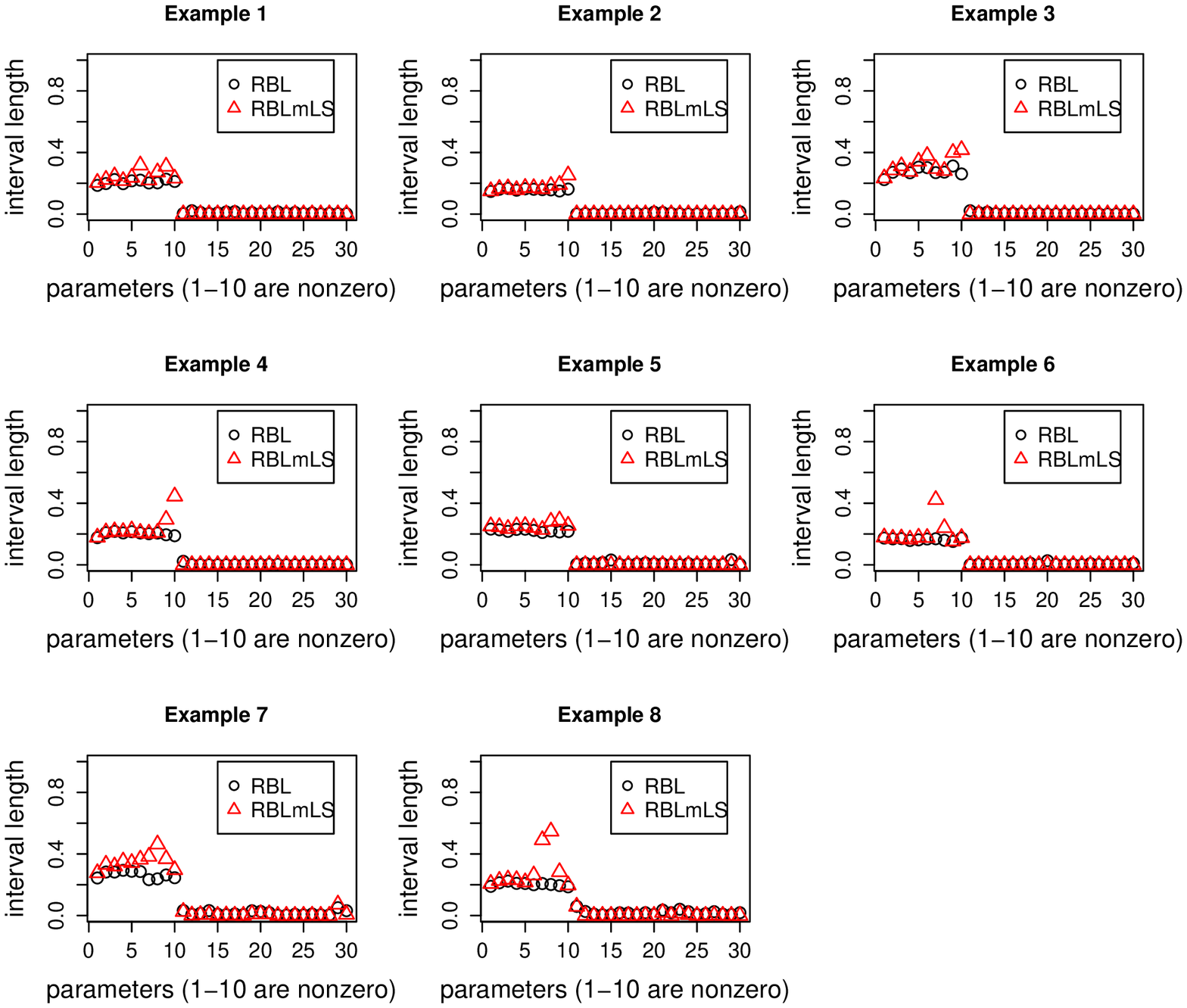}
\caption{Lengths of $90\%$ confidence intervals for each $\beta^*_j,j=1,..,s,s+1,...,s+20$ based on two methods: residual bootstrap Lasso+mLS (RBLmLS, red triangle) and residual bootstrap Lasso (RBL, black circle). For a better view, only $20$ zero-valued parameters are present. The interval lengths for the remaining $p-s-20$ zero-valued parameters are similar to the $20$ presented and therefore are omitted.}
\label{fig:length}
\end{figure}

In practice, many people prefer to perform paired bootstrap Lasso (resampling from the pairs $(x_i,y_i),i=1,...,n$ instead of from the residual) even when it makes sense to think of the design matrix as fixed. Therefore, we give some comparisons of residual bootstrap Lasso+mLS and paired bootstrap Lasso. Figure~\ref{fig:problength} shows the coverage probabilities v.s. average interval lengths based on residual bootstrap Lasso+mLS and paired bootstrap Lasso for different examples. Again, we average the coverage probabilities and interval lengths over nonzero-valued parameters part and zero-valued parameters part respectively and show them in Table~\ref{tab:meancp} and Table~\ref{tab:meanil}. We can see that residual bootstrap Lasso+mLS provides more accurate coverage probabilities ($0.5\%$ closer to $1$ for zero $\beta_j^*$ and $14\%$ in average closer to the preassigned level for nonzero $\beta_j^*$) with more than $90\%$ shorter (for zero $\beta_j^*$) or at least comparable (for nonzero $\beta_j^*$) interval lengths compared with paired bootstrap Lasso. Based on our simulations, we conclude that residual bootstrap Lasso+mLS and residual bootstrap Lasso+Ridge are better choices for constructing confidence intervals.

\begin{figure}[htbp]
\centering\includegraphics[width=5in]{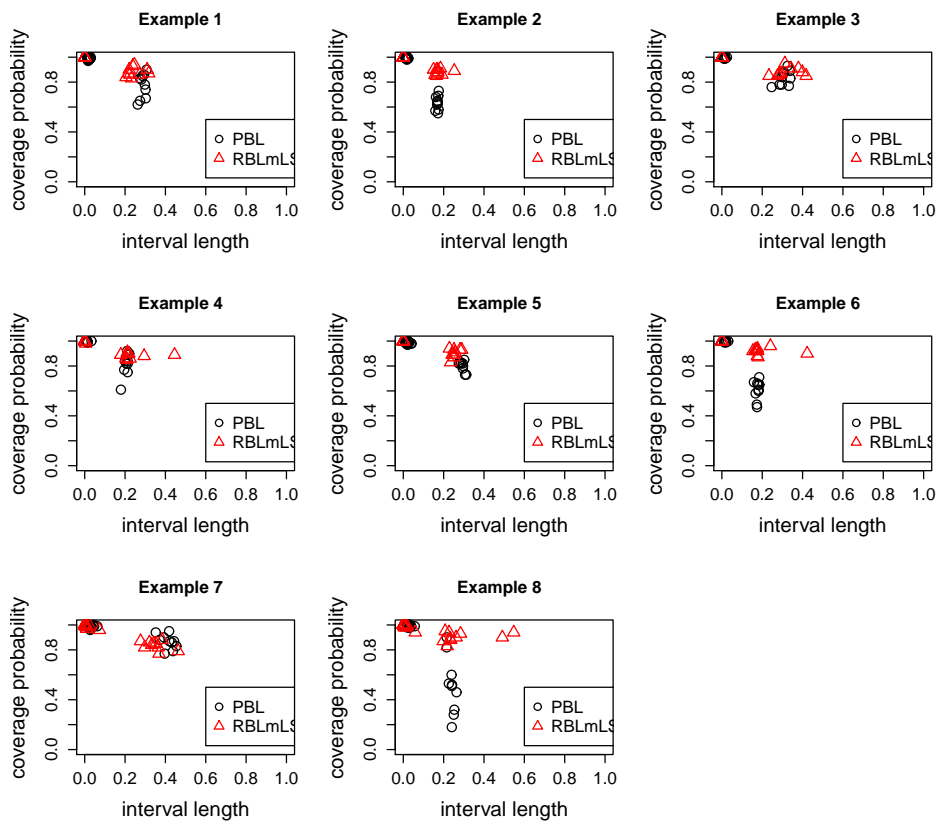}
\caption{Coverage probabilities v.s. average interval lengths for $90\%$ confidence intervals based on two methods: paired bootstrap Lasso (PBL, black circle) and residual bootstrap Lasso+mLS (RBLmLS, red triangle). The latter is better since it provides more accurate coverage probabilities ($0.5\%$ closer to $1$ for zero $\beta_j^*$ and $14\%$ in average closer to the preassigned level ($90\%$) for nonzero $\beta_j^*$) but with $90\%$ shorter (for zero $\beta_j^*$) or at least comparable (for nonzero $\beta_j^*$) interval lengths.}
\label{fig:problength}
\end{figure}

\section{Conclusion}

We have derived for the first time the asymptotic properties of Lasso+mLS and Lasso+Ridge in sparse high-dimensional linear regression models where $p \gg n$. Under the Irrepresentable condition and other common conditions on scaling $(n,s,p)$, we showed that both Lasso+mLS and Lasso+Ridge are asymptotically unbiased and they achieve oracle convergence rate of $\frac{s}{n}$ for MSE which improves the performance of the Lasso. In addition, Lasso+mLS and Lasso+Ridge estimators have an oracle property in the sense that they can select the true predictors with probability converging to 1 and the estimates of nonzero parameters have the same asymptotic normal distribution that they would have if the zero parameters were known in advance.

We then proposed residual bootstrap after Lasso+mLS and Lasso+Ridge methods and showed that they give valid approximations to the distributions of Lasso+mLS and Lasso+Ridge, respectively, provided that the probability of the Lasso selecting wrong models decays at an exponential rate and the number of true predictors $s$ goes to infinity slower than $\sqrt{n}$. In fact, our analysis is not limited to adopting Lasso in the selection stage, but is applicable to any other model selection criteria with exponentially decay rates of the probability of selecting wrong models, for example, stability selection, SCAD and Dantzig selector.

Lastly, we presented simulation results assessing the finite sample performance of Lasso+mLS and Lasso+Ridge and observe that they not only dramatically decrease the bias$^2$ of the Lasso by more than $90\%$ but also reduce the MSE and PMSE by $40\%-80\%$ and $5\%-25\%$ respectively. Further, we constructed $90\%$ confidence interval based on our residual bootstrap Lasso+mLS (or Lasso+Ridge) and examined the coverage accuracy. We found that our method resulted in coverage probability approximately $88\%$ for nonzero $\beta^*_j$ and $1$ for zero $\beta^*_j$, which is much more accurate (approximately $7\%$ closer to the desired level) than bootstrap Lasso method.

\section*{Acknowledgements}
Hanzhong Liu would like to thank the China Scholarship Council (CSC) for financial support and the Statistics Department at UC Berkeley for hosting his visit during which this work was finished. This research is also supported in part by NSF grants SES-0835531 (CDI), DMS-1107000, DMS-1228246, ARO grant W911NF-11-1-0114, and the Center for Science of Information (CSoI), an US NSF Science and Technology Center, under grant agreement CCF-0939370. The authors are very grateful to Taesup Moon, Siqi Wu, Jyothsna Sainath, Adam Bloniarz and the referees for many insightful comments and helpful suggestions that lead to a substantially improved manuscript.

\appendix

\section{Model Selection Consistency of the Lasso}\label{A}

\begin{lemma}[Zhao and Yu (2006)]
\label{lemma: model selection}
Under conditions (a)-(d), (i)-(j) and the Irrepresentable Condition $\eqref{eqn:IC}$, the Lasso has strong sign consistency. That is,
\[ P(sign(\hat \beta(\lambda_n)) = sign( \beta^*)) \geq 1- o(e^{-n^{c_2}})\rightarrow 1 \ as\ n\rightarrow \infty. \]
\end{lemma}
\ \\
{\bf Remark A.1.} In fact, looking carefully through the proof of Lemma~\ref{lemma: model selection} in \cite{ZhaoYu2006}, the Gaussian assumption (a) can be relaxed by a subgaussian assumption. That is, assume that there exists constants $C$, $c>0$ so
\[ P(|\epsilon_i|\geq t)\leq C e^{-ct^2}, \ \forall t\geq 0. \]

This result tells us that using the Lasso we can allow $p$ to grow faster than $n$ (up to exponentially fast) while the probability of correct model selection still converges to 1 fast.

\section{Stability Selection}\label{B}

Here we will give a brief introduction of stability selection combined with randomized Lasso in sparse linear regression.

The randomized Lasso is a generalization of the Lasso, which penalizes the absolute value $|\beta_k|$ of every component with a penalty randomly chosen in the range $[\lambda, \lambda/\alpha]$ for $\alpha \in (0,1]$. Let $W_k$ be i.i.d. random variables in $[\alpha,1]$ for $k=1,...,p$. The randomized Lasso estimator $\hat \beta^{\lambda,W}$ is defined by
\begin{equation}
 \hat \beta^{\lambda,W} = \mathop {argmin}\limits_{\beta} \left\{  || {Y - X\beta } ||^2  + \lambda \sum_{k=1}^{p} \frac{|\beta_k|}{W_k} \right\}
\end{equation}
where $\lambda \in R^+$ is the regularization parameter. \cite{MeinshausenBuhlmann2010} proposed an appropriate distribution of the weights $W_k$: $W_k=\alpha$ with probability $p_w \in (0,1)$ and $W_k=1$ otherwise.

For any given regularization parameter $\lambda \in \Lambda \subseteq R^+$, denote the selected predictor set based on the samples $I\subset\{1,...,n\}$ as $\hat S^{\lambda,W}(I) = \{k: \hat \beta^{\lambda,W}_k \neq 0 \}$.

\begin{definition}[Meinshausen and Buhlmann (2010)]
Let $I$ be a random subsample of $\{1,...,n\}$ of size $\lfloor n/2 \rfloor$, drawn with replacement. For every set $K \subseteq \{1,...,p\}$, the probability of being in the selected set $\hat S^{\lambda,W}(I)$ is
\begin{equation}
 \hat \Pi^{\lambda}_K = P^*\{ K \subseteq \hat S^{\lambda,W}(I) \}
\end{equation}
where the probability $P^*$ is with respect to both the random subsampling and the randomness of the weights $W_k$.
\end{definition}
With stability selection, we subsample the data many times and then choose the predictors with a high selection probability.
\begin{definition}[Meinshausen and Buhlmann (2010)]
For a cut-off $\pi_{thr}$ with $0< \pi_{thr} <1$ and a set of regularization parameters $\Lambda$, the set of stable variables is defined as
\begin{equation}
 \hat S^{stable} = \{k: \mathop {max}\limits_{\lambda \in \Lambda}\hat \Pi^{\lambda}_k \geq \pi_{thr} \}.
\end{equation}
\end{definition}
For stability selection with randomized Lasso, one can obtain selection consistency by just assuming sparse eigenvalues, a condition that is much weaker than that of the Irrepresentability. Sparse eigenvalues condition essentially requires that the minimum and maximum eigenvalues, for a selection of order $s$ predictors, are bounded away from 0 and $\infty$ respectively.

\begin{definition}[Meinshausen and Buhlmann (2010)]
For any $K\subseteq \{1,...,p\}$, let $X_K$ be the restriction of $X$ to columns in $K$. The minimal sparse eigenvalue $\phi_{min}$ is defined for $k\leq p$ as
\begin{equation}
 \phi_{min}(k) = \mathop {inf}\limits_{a \in R^{\lceil k \rceil},K\subseteq \{1,...,p\}:|K|\leq \lceil k \rceil} \left\{  \frac{||X_Ka||}{||a||} \right\}
\end{equation}
and analogously for the maximal sparse eigenvalue $\phi_{max}$.
\end{definition}
\ \\
{\em Sparse eigenvalues assumption: there are some $C>1$ and some $\kappa \geq 9$ such that
\[ \frac{\phi_{max}(Cs^2)}{\phi_{min}^{3/2}(Cs^2)} < \sqrt{C}/\kappa. \] }

Under this assumption, we can state the selection consistency result obtained directly from Theorem 2 in \cite{MeinshausenBuhlmann2010}.
\begin{lemma}[Meinshausen and Buhlmann (2010)]
\label{lemma: model selection stability}
For the randomized Lasso, let $\alpha$ be given by $\alpha^2= \nu \phi_{min}(m)/m$, for any $\nu \in ((7/\kappa)^2,1/\sqrt{2})$, and $m=Cs^2$. Let $c_2<c<1$ and $\lambda_{min} = 2 \sigma \{\sqrt{2C}s+1\}n^{(c-1)/2}$. Assume that $p=O(e^{n^{c_2}})>10$ and $s\geq 7$ and that the gaussian assumption (a) and the sparse eigenvalues assumption are satisfied. For any $\lambda \geq \lambda_{min}$, if $\mathop {min}\limits_{1\leq i \leq s}|\beta_i^*| \geq (Cs)^{3/2}(0.3\lambda)$, then there is some $\delta \in (0,1)$ such that, for all $\pi_{thr}\geq 1-\delta$, stability selection with randomized Lasso satisfies
$$P(\hat S^{stable}_{\lambda} = S) \geq 1- o(e^{-n^{c_2}})$$
where $\hat S^{stable}_{\lambda}=\{ k: \Pi^{\lambda}_k \geq \pi_{thr} \}$.
\end{lemma}

\section{Technical details}\label{C}

\begin{proof}[Proof of Theorem~\ref{theorem: asymptotic bias and variance}] Denote $\tilde \beta$ the Select+mLS. Conditioned on $\{\hat S=S\}$, for $\tau_n^2 \leq \Lambda_{min}(C_{11})$, we have
$$\tilde \beta_S = (X_{S}^TX_{S})^{-1}X_{S}^TY = \beta^*_S + (X_{S}^TX_{S})^{-1}X_{S}^T \epsilon. $$
Combine with triangle inequality,
\begin{eqnarray}
  ||E \tilde \beta - \beta^*||_2 & \leq & || E \tilde \beta \mathbb{I}_{\hat S =S}  - \beta^*||_2 + || E \tilde \beta \mathbb{I}_{\hat S \neq S}||_2 \nonumber \\
 & = & || E \{(X_{S}^TX_{S})^{-1}X_{S}^TY \mathbb{I}_{\hat S =S} \} - \beta^*||_2 + || E \tilde \beta \mathbb{I}_{\hat S \neq S} ||_2   \nonumber \\
 & \leq & ||E \{(X_{S}^TX_{S})^{-1}X_{S}^TY  - \beta^*_S \}||_2  + || E \{(X_{S}^TX_{S})^{-1}X_{S}^TY \mathbb{I}_{\hat S \neq S} \} ||_2  \nonumber \\
 & & + || E \tilde \beta \mathbb{I}_{\hat S \neq S} ||_2 \nonumber \\
 & = & || E \{(X_{S}^TX_{S})^{-1}X_{S}^TY \mathbb{I}_{\hat S \neq S} \} ||_2 + || E \tilde \beta \mathbb{I}_{\hat S \neq S} ||_2 \nonumber
 \label{eqn:1}
\end{eqnarray}
where $\mathbb{I}_A$ is the indicater function. The last equality holds since
$$E \{(X_{S}^TX_{S})^{-1}X_{S}^TY\} =  \beta^*_S + (X_{S}^TX_{S})^{-1}X_{S}^T E\epsilon = \beta^*_S. $$
By Cauchy-Schwarz inequality,
\[
 || E \{(X_{S}^TX_{S})^{-1}X_{S}^TY \mathbb{I}_{\hat S \neq S} \} ||_2^2 \leq  E ||\{(X_{S}^TX_{S})^{-1}X_{S}^TY\}||_2^2 P (\hat S \neq S),
\]
\[
 || E \tilde \beta \mathbb{I}_{\hat S \neq S} ||_2^2 \leq  E ||\tilde \beta||_2^2 P (\hat S \neq S).
\]
So we need to control $E ||\{(X_{S}^TX_{S})^{-1}X_{S}^TY\}||_2^2$ and $E ||\tilde \beta||_2^2$ respectively.
\begin{eqnarray}
 E ||\{(X_{S}^TX_{S})^{-1}X_{S}^TY\}||_2^2 & = & E|| \beta^*_S + (X_{S}^TX_{S})^{-1}X_{S}^T \epsilon ||_2^2 \nonumber \\
 & = & ||\beta^*_S||_2^2 +  E || (X_{S}^TX_{S})^{-1}X_{S}^T \epsilon ||_2^2    \nonumber \\
 & = & ||\beta^*||_2^2 +  \sigma^2 tr(X_{S}^TX_{S})^{-1} \nonumber \\
 & = & ||\beta^*||_2^2 + \frac{\sigma^2}{n} tr(C_{11}^{-1}). \nonumber
\end{eqnarray}
By definition $\eqref{eqn:modifiedols}$, we have
\begin{eqnarray}
 ||\tilde \beta||_2^2  = || \frac{1}{\sqrt{n}} V \tilde D U^T Y ||_2^2  =  \frac{1}{n} Y^T U \tilde D^T \tilde D U^T Y   \leq   \frac{1}{n} \tau_n^{-2} ||Y||_2^2,
 \label{eqn:tiltebetabound}
\end{eqnarray}
where $\tilde{D}$ is a $d\times n$ diagonal matrix with diagonal entries $\lambda_1^{-1}$, $\lambda_2^{-1}$,...,$\lambda_d^{-1}$ (Note that we take $\lambda_k^{-1}=0$ for all $\lambda_k< \tau_n$). The last equality holds since the largest singular value of $\tilde D$ is no more than $\tau_n^{-1}$ and $U$ is an orthogonal matrix. Moreover,
\begin{eqnarray}
 E||Y||_2^2  = E||  X \beta^* +\epsilon ||_2^2 =   ||X\beta^*||_2^2 + n\sigma^2.
\end{eqnarray}
Combining the above results, we obtain $\eqref{eqn:theo1}$.

Next, we prove the second part of Theorem~\ref{theorem: asymptotic bias and variance}.
\begin{eqnarray}
 E ||\tilde \beta - \beta^*||_2^2 & = & E ||\tilde \beta - \beta^*||_2^2\mathbb{I}_{\hat S =S} + E ||\tilde \beta - \beta^*||_2^2\mathbb{I}_{\hat S\neq S} \nonumber \\
 & = &  E || \{(X_{S}^TX_{S})^{-1}X_{S}^T \epsilon \}||_2^2 \mathbb{I}_{\hat S =S}  + E ||\tilde \beta - \beta^*||_2^2\mathbb{I}_{\hat S\neq S}   \nonumber \\
 & \leq & E || \{(X_{S}^TX_{S})^{-1}X_{S}^T \epsilon \}||_2^2 +2 ( E ||\tilde \beta||_2^2 \mathbb{I}_{\hat S\neq S} + E ||\beta^*||_2^2 \mathbb{I}_{\hat S\neq S} ) \nonumber \\
 & \leq & \frac{\sigma^2}{n} tr(C_{11}^{-1}) + 2 \sqrt{ P(\hat S\neq S) } (\sqrt{ E||\tilde \beta||_2^4} + ||\beta^*||_2^2).
 \label{eqn:1}
\end{eqnarray}
The last inequality holds for Cauchy-Schwarz inequality. Since $Y=X \beta^* + \epsilon$, we have
\[ ||Y||_2^4 = (||Y||_2^2)^2\leq 4(||X \beta^*||_2^2+||\epsilon||_2^2)^2 \leq 8 (||X \beta^*||_2^4+||\epsilon||_2^4). \]
Because $\epsilon_i \sim N(0,\sigma^2),\ i.i.d.$, we get $E(\epsilon_i)^4 = 3 \sigma^4$ and
\begin{equation}
\label{eqn:useinlemma3}
E||\epsilon||_2^4 = E(\sum\limits_{i=1}^{n} \epsilon_i^2)^2 = \sum\limits_{i=1}^{n} E\epsilon_i^4 + 2\sum\limits_{i<j}E\epsilon_i^2\epsilon_j^2 = (n^2+2n) \sigma^4,
\end{equation}
hence when $n\geq 2$
\[ E||Y||_2^4 \leq 8 (||X \beta^*||_2^4+ E||\epsilon||_2^4) \leq 16 (||X \beta^*||_2^4+ n^2 \sigma^4). \]
Connect $\eqref{eqn:tiltebetabound}$, then
\begin{eqnarray}
 E||\tilde \beta||_2^4 \leq  \frac{1}{n^2} \tau_n^{-4} E||Y||_2^4 \leq 16 \tau_n^{-4} (\frac{1}{n^2}||X\beta^*||_2^4 + \sigma^4 ).
 \label{eqn:beta4}
\end{eqnarray}
Taking $\eqref{eqn:beta4}$ back to $\eqref{eqn:1}$ gives the result.
\end{proof}
\ \\

\begin{proof}[Proof of Theorem~\ref{theorem: asymptotic bias and variance ridge}] Denote $\tilde \beta$ the Select+Ridge, we have
\begin{equation}
 \tilde \beta_{\hat S} =(X_{\hat S}^TX_{\hat S}+ \mu_n I)^{-1}X_{\hat S}^TY. \ 
 \label{eqn:lassoridge}
\end{equation}
Applying SVD decomposition of $\frac{1}{\sqrt{n}} X_{\hat S}$ in $\eqref{eqn:SVD}$, it is easy to obtain
\begin{equation}
 \tilde \beta_{\hat S} = VD_1D^TU^TY
 \label{eqn:lassoridge}
\end{equation}
where $D_1$ is a diagonal matrix with diagonal entries $\frac{\sqrt{n}}{n\lambda_1^2+\mu_n},...,\frac{\sqrt{n}}{n\lambda_d^2+\mu_n}$. Since $V$ is an orthogonal matrix,
\[ ||\tilde \beta_{\hat S}||_2^2= Y^T U D D_1^T V^T V D_1 D^T U^T Y = Y^T U D_2 U^T Y \]
where
\[ D_2 = diag \{ \frac{n\lambda_1^2}{(n\lambda_1^2+\mu_n)^2},...,\frac{n\lambda_d^2}{(n\lambda_d^2+\mu_n)^2} \}. \]
Therefore,
\[ ||\tilde \beta_{\hat S}||_2^2 \leq \Lambda_{max}(D_2) Y^T U U^T Y =  \Lambda_{max}(D_2) ||Y||_2^2 \leq \frac{1}{4 \mu_n} ||Y||_2^2, \]
the last inequality is due to $\frac{n\lambda_i^2}{(n\lambda_i^2+\mu_n)^2} \leq \frac{1}{4 \mu_n},\ i=1,...,d$. Then,
\[ E||\tilde \beta_{\hat S}||_2^2 \leq \frac{1}{4 \mu_n} E||Y||_2^2 = \frac{1}{4 \mu_n} ( ||X\beta^*||_2^2 + n\sigma^2), \]
\[ E||\tilde \beta_{\hat S}||_2^4 \leq \frac{1}{16 \mu_n^2} E||Y||_2^4 = \frac{1}{16 \mu_n^2} 16 (||X \beta^*||_2^4+ n^2 \sigma^4). \]
Combine Cauchy-Schwarz inequality, we have
\[ || E \tilde \beta_{\hat S} \mathbb{I}_{\hat S \neq S} ||_2^2 \leq P(\hat S \neq S) \frac{n}{4 \mu_n} ( \frac{1}{n} ||X\beta^*||_2^2 + \sigma^2), \]
\[ || E \tilde \beta_S \mathbb{I}_{\hat S \neq S} ||_2^2 \leq P(\hat S \neq S) \frac{n}{4 \mu_n} ( \frac{1}{n} ||X\beta^*||_2^2 + \sigma^2). \]
Moreover,
\begin{eqnarray}
 \tilde \beta_S - \beta^*_S & = & VD_1D^TU^TY - \beta^*_S = (VD_1D^TU^T  X_S -I ) \beta^*_S +  VD_1D^TU^T \epsilon  \nonumber \\
 & = & ( \sqrt{n} VD_1D^TU^T UDV^T -I )\beta^*_S +  VD_1D^TU^T \epsilon  \nonumber \\
 & = &   V D_3 V^T \beta^*_S + VD_1D^TU^T \epsilon  \nonumber
\end{eqnarray}
where
\[ D_3 =  diag\{ \frac{-\mu_n}{(n\lambda_1^2+\mu_n)},...,\frac{-\mu_n}{(n\lambda_s^2+\mu_n)} \}, \]
and using $(\frac{\mu_n}{n\lambda_i^2+\mu_n})^2 \leq \frac{\mu_n^2}{n^2 \Lambda_{min}^2},\ i=1,...,s $, we have
\[ ||E\tilde \beta_{S}- \beta^*_S||_2^2 = || V D_3 V^T \beta^*_S ||_2^2 \leq \frac{\mu_n^2}{n^2 \Lambda_{min}^2} ||\beta^*||_2^2. \]
Therefore,
\begin{eqnarray}
 ||E \tilde \beta - \beta^*||_2 & \leq & || E \tilde \beta \mathbb{I}_{\hat S =S}  - \beta^*||_2 + || E \tilde \beta \mathbb{I}_{\hat S \neq S}||_2 \nonumber \\
 & \leq &  ||E\tilde \beta_{S}- \beta^*_S||_2 + || E \tilde \beta_S \mathbb{I}_{\hat S \neq S}  ||_2 + || E \tilde \beta_{\hat S} \mathbb{I}_{\hat S \neq S} ||_2  \nonumber \\
 & \leq & \frac{\mu_n}{n \Lambda_{min}} ||\beta^*||_2 + 2 \sqrt{P(\hat S \neq S) \frac{n}{4 \mu_n} ( \frac{1}{n} ||X\beta^*||_2^2 + \sigma^2) }, \nonumber
\end{eqnarray}
which proves the first part of Theorem~\ref{theorem: asymptotic bias and variance ridge}. For the second part, we have
\begin{eqnarray}
 E ||\tilde \beta_S - \beta^*_S||_2^2 & = & ||V D_3 V^T \beta^*_S||_2^2 + E||VD_1D^TU^T \epsilon||_2^2 \nonumber \\
 & \leq & \frac{\mu_n^2}{n^2 \Lambda_{min}^2} ||\beta^*||_2^2 + \sigma^2 tr(VD_1D^TU^TUDD_1V^T)  \nonumber \\
 & = & \frac{\mu_n^2}{n^2 \Lambda_{min}^2} ||\beta^*||_2^2 + \sigma^2 \sum\limits_{i=1}^{s}  \frac{n\lambda_i^2}{(n\lambda_i^2+\mu_n)^2}. \nonumber
\end{eqnarray}
Algebraic operation yields
\[ \sigma^2 \sum\limits_{i=1}^{s}  \frac{n\lambda_i^2}{(n\lambda_i^2+\mu_n)^2}= \frac{\sigma^2}{n} tr\{(C_{11}+\frac{\mu_n}{n}I)^{-2}C_{11}\}, \]
then,
\[ E ||\tilde \beta_S - \beta^*_S||_2^2 \leq \frac{\sigma^2}{n} tr\{(C_{11}+\frac{\mu_n}{n}I)^{-2}C_{11}\} + \frac{\mu_n^2}{n^2 \Lambda_{min}^2} ||\beta^*||_2^2. \]
On the other hand,
\begin{eqnarray}
 E ||\tilde \beta - \beta^*||_2^2\mathbb{I}_{\hat S\neq S} & \leq  & 2 \sqrt{ P(\hat S\neq S) } (\sqrt{ E||\tilde \beta||_2^4} + ||\beta^*||_2^2) \nonumber \\
 & \leq &  2 \sqrt{ P(\hat S \neq S) }  \{ ||\beta^*||_2^2 + \frac{n}{\mu_n} [\frac{1}{n}||X\beta^*||_2^2  + \sigma^2]  \}.  \nonumber
\end{eqnarray}
Applying the same trick as $\eqref{eqn:1}$, we obtain the result.
\end{proof}
\ \\

\begin{proof}[Proof of Theorem~\ref{theorem: asymptotic normality}] We prove the results for Select+mLS and Select+Ridge respectively.

(1) Select+mLS: as stated in the proof of Theorem~\ref{theorem: asymptotic bias and variance}, conditioned on $\{\hat S=S\}$, when $n$ is large enough, $\tau_n^2 \propto \frac{1}{n^2} \leq \Lambda_{min} \leq \Lambda_{min}(C_{11})$, then Select+mLS $\tilde \beta$ satisfies:
$$\tilde \beta_S = (X_{S}^TX_{S})^{-1}X_{S}^TY = \beta^*_S + (X_{S}^TX_{S})^{-1}X_{S}^T \epsilon.$$
Because $\epsilon_i \sim N(0,\sigma^2),\ i.i.d.$, then $\sqrt{n}(X_{S}^TX_{S})^{-1}X_{S}^T \epsilon \sim N(0,\sigma^2 n(X_{S}^TX_{S})^{-1})$, which is $N(0,\sigma^2 C_{11}^{-1})$ since $C_{11} =\frac{1}{n}X_{S}^TX_{S} $. Therefore,
\begin{eqnarray}
  \hat \Psi(t) & = & P(\sqrt{n}( \tilde \beta_S - \beta^*_S ) \leq t) \nonumber \\
 & = & P(\sqrt{n}( \tilde \beta_S - \beta^*_S ) \leq t, \hat S =S) + P(\sqrt{n}( \tilde \beta_S - \beta^*_S )\leq t,\ \hat S\neq S )  \nonumber \\
 & = & P(\sqrt{n}(X_{S}^TX_{S})^{-1}X_{S}^T \epsilon \leq t,\ \hat S =S ) + P(\sqrt{n}( \tilde \beta_S - \beta^*_S )\leq t,\ \hat S\neq S ) \nonumber \\
 & = & \Psi(t) - P(\sqrt{n}(X_{S}^TX_{S})^{-1}X_{S}^T \epsilon \leq t,\ \hat S \neq S ) \nonumber \\
 & & + P(\sqrt{n}( \tilde \beta_S - \beta^*_S )\leq t,\ \hat S\neq S ). \nonumber
 \label{eqn:111}
\end{eqnarray}
Then
\begin{eqnarray}
  & &\mathop {sup}\limits_{t \in R^s} | \hat \Psi(t) - \Psi(t)|  \nonumber \\
 & \leq & \mathop {sup}\limits_{t \in R^s} \{ P(\sqrt{n}(X_{S}^TX_{S})^{-1}X_{S}^T \epsilon \leq t,\ \hat S \neq S ) + P(\sqrt{n}( \tilde \beta_S - \beta^*_S )\leq t,\ \hat S\neq S ) \} \nonumber \\
 & \leq & 2 P(\ \hat S \neq S ) \rightarrow 0, \ as\ n\rightarrow \infty. \nonumber
 \label{eqn:112}
\end{eqnarray}

(2) Select+Ridge: again, conditioned on $\{\hat S=S\}$, the Select+Ridge and Select+mLS are respectively
\[ \tilde \beta_{Select+Ridge,S} = VD_1D^TU^TY, \]
\[ \tilde \beta_{Select+mLS,S} = (X_{S}^TX_{S})^{-1}X_{S}^TY = \frac{1}{\sqrt{n}} V (D^TD)^{-1}D^TU^TY. \]
By simple calculation, the difference between these two estimators is
\begin{eqnarray}
 & & ||\sqrt{n}(\tilde \beta_{Select+Ridge,S} - \tilde \beta_{Select+mLS,S})||_2^2 \nonumber \\
 & = & n\cdot ||V diag\{\frac{-\sqrt{n}\mu_n}{n\lambda_1^2(n\lambda_1^2 +\mu_n)},... \frac{-\sqrt{n}\mu_n}{n\lambda_s^2(n\lambda_s^2 +\mu_n)} \} D^TU^TY||_2^2   \nonumber \\
 & \leq & \mathop {max}\limits_{1\leq i \leq s} \{ \frac{n \mu_n^2}{n\lambda_i^2(n\lambda_i^2 +\mu_n)^2} \} ||Y||_2^2 \nonumber \\
 & \leq & \frac{\mu_n^2}{n^2 \Lambda_{min}^3} ||Y||_2^2 = O_p(\mu_n^2)
 \label{eqn:112}
\end{eqnarray}
where the last equality comes from assumption $(g)$ and $E||Y||_2^2 = ||X\beta^*||_2^2+n\sigma^2=O(n^2)$. Therefore, Select+Ridge has the same asymptotic distribution as Select+mLS, which completes the proof.
\end{proof}
\ \\

In the following, we will prove the validity of residual bootstrap after Select+mLS. The proof for residual bootstrap after Select+Ridge is omitted since the techniques are almost the same.

Firstly, we re-characterize the conditional distribution of bootstrap error terms $\epsilon_i^*,i=1,...n$ as follows:
\begin{lemma}
\label{lemma:subgaussian}
Suppose that assumptions (a)-(e), (h) and (dd) are satisfied and that $P(\hat S \neq S) \rightarrow 0$, then with probability converging to 1, $\epsilon_i^*,i=1,...,n$ are conditionally i.i.d. subgaussian random variables. That is, there exists constant $C^*$, $c^*>0$ such that
\begin{equation}
\label{eqn:subguassionbootstrap}
P^*(|\epsilon_i^*| \geq t) \leq C^* e^{-c^*t^2}, \ \forall t\geq 0
\end{equation}
holds in probability.
\end{lemma}

\begin{proof} Note that $P^*(|\epsilon_i^*| \geq t) = \frac{1}{n} \sum\limits_{i=1}^{n} \mathbb{I}_{|\hat \epsilon_i - \hat \mu| \geq t}$, hence $\eqref{eqn:subguassionbootstrap}$ is equivalent to
\begin{equation}
 \label{eqn:subguassionbootstrap1}
 \mathop {sup}\limits_{t\geq 0} \{ \frac{1}{n} \sum\limits_{i=1}^{n} e^{c^*t^2} \mathbb{I}_{|\hat \epsilon_i - \hat \mu| \geq t} \} \leq C^*.
\end{equation}
We know that
\begin{eqnarray}
  \hat \epsilon_i - \hat \mu & = & y_i - x_i^T \tilde \beta - (\overline{y} - \overline{x}^T \tilde \beta) \nonumber \\
 & = & x_i^T \beta^* + \epsilon_i - x_i^T \tilde \beta - (\overline{x}^T \beta^* + \overline{\epsilon} - \overline{x}^T \tilde \beta) \nonumber \\
 & = & x_i ^T (\beta^* - \tilde \beta) + \epsilon_i - \overline{\epsilon} \nonumber
 \label{eqn:1121}
\end{eqnarray}
where $x_i^T$ is the $i$-th row of $X$, $\overline{y}=\frac{1}{n} \sum\limits_{i=1}^{n} {y_i}$,$\overline{\epsilon}= \frac{1}{n} \sum\limits_{i=1}^{n} {\epsilon_i}$ and $\overline{x}= \frac{1}{n} \sum\limits_{i=1}^{n} {x_i} = 0$.
It is easy to see that $\mathop {sup}\limits_{t\geq 0} \{ \frac{1}{n} \sum\limits_{i=1}^{n} e^{c^*t^2} \mathbb{I}_{|\hat \epsilon_i - \hat \mu| \geq t} \}$ can be bounded by
\begin{eqnarray}
 \frac{1}{n} \sum\limits_{i=1}^{n} \left \{ \mathop {sup}\limits_{t\geq 0} \{ e^{c^*t^2} \mathbb{I}_{|x_i^T (\beta^* - \tilde \beta)| \geq t/3} \} + \mathop {sup}\limits_{t\geq 0} \{ e^{c^*t^2} \mathbb{I}_{|\overline{\epsilon}| \geq t/3} \} + \mathop {sup}\limits_{t\geq 0} \{ e^{c^*t^2} \mathbb{I}_{|\epsilon_i| \geq t/3} \} \right \}. \nonumber \\
 \label{eqn:1121}
\end{eqnarray}

For the first term in $\eqref{eqn:1121}$, let $x_{i,S} = (x_{i1},...,x_{is})^T$,
\begin{eqnarray}
  & & P(\mathop {max}\limits_{1\leq i \leq n} |x_i^T (\beta^* - \tilde \beta)| \geq 1/3) \nonumber \\
  & = & P(\mathop {max}\limits_{1\leq i \leq n} |x_i^T (\beta^* - \tilde \beta)| \geq 1/3,\hat S=S) \nonumber \\
  & & + P(\mathop {max}\limits_{1\leq i \leq n} |x_i^T (\beta^* - \tilde \beta)| \geq 1/3,\hat S \neq S)  \nonumber \\
 & \leq & P(\mathop {max}\limits_{1\leq i \leq n} |x_{i,S}^T(X_{S}^TX_{S})^{-1}X_{S}^T \epsilon | \geq 1/3) + P(\hat S \neq S).
\end{eqnarray}
By Markov inequality,
\[ P(\mathop {max}\limits_{1\leq i \leq n} |x_{i,S}^T(X_{S}^TX_{S})^{-1}X_{S}^T \epsilon | \geq 1/3) \leq 9 E\mathop {max}\limits_{1\leq i \leq n}|x_{i,S}^T(X_{S}^TX_{S})^{-1}X_{S}^T \epsilon |^2.  \]
Note that
\[ \mathop {max}\limits_{1\leq i \leq n}|x_{i,S}^T(X_{S}^TX_{S})^{-1}X_{S}^T \epsilon |^2 \leq \mathop {max}\limits_{1\leq i \leq n}||x_{i,S}||^2 \cdot ||(X_{S}^TX_{S})^{-1}X_{S}^T \epsilon ||^2, \]
therefore
\[ P(\mathop {max}\limits_{1\leq i \leq n} |x_{i,S}^T(X_{S}^TX_{S})^{-1}X_{S}^T \epsilon | \geq 1/3) \leq 9 \mathop {max}\limits_{1\leq i \leq n}||x_{i,S}||^2  \cdot E ||(X_{S}^TX_{S})^{-1}X_{S}^T \epsilon ||^2.  \]
Because $\epsilon \sim N(0,\sigma^2 I)$, hence $(X_{S}^TX_{S})^{-1}X_{S}^T \epsilon \sim N(0,\sigma^2 (X_{S}^TX_{S})^{-1})$. Then,
\[ E||(X_{S}^TX_{S})^{-1}X_{S}^T \epsilon ||^2= \sigma^2 tr(X_{S}^TX_{S})^{-1} =\sigma^2 \frac{1}{n} tr(C_{11}^{-1}) \leq \frac{\sigma^2}{\Lambda_{min}} \frac{s}{n}, \]
hence
\[ P(\mathop {max}\limits_{1\leq i \leq n} |x_{i,S}^T(X_{S}^TX_{S})^{-1}X_{S}^T \epsilon | \geq 1/3) \leq 9 \frac{\sigma^2}{\Lambda_{min}} \frac{s}{n} \mathop {max}\limits_{1\leq i \leq n}||x_{i,S}||^2 \rightarrow 0 \]
where "$\rightarrow 0$" comes from the assumptions $s^2/n \rightarrow 0$ and $\mathop {max}\limits_{1\leq i \leq n}||x_{i,S}||^2 = \mathop {max}\limits_{1\leq i \leq n}\sum\limits_{j=1}^{s}x_{ij}^2=o(n^{\frac{1}{2}})$.
Connect with $P(\hat S \neq S) \rightarrow 0$, we have
\[ P(\mathop {max}\limits_{1\leq i \leq n} |x_i^T (\beta^* - \tilde \beta)| \geq 1/3) \rightarrow 0, \]
hence
\[ P(\frac{1}{n} \sum\limits_{i=1}^{n} \mathop {sup}\limits_{t\geq 1} \{ e^{c^*t^2} \mathbb{I}_{|x_i^T (\beta^* - \tilde \beta)| \geq t/3} \} \leq  e^{c^*} ) \geq  P(\mathop {max}\limits_{1\leq i \leq n} |x_i^T (\beta^* - \tilde \beta)| < 1/3) \rightarrow 1. \]
The inequality holds since it is easy to show that
\[ \{\frac{1}{n} \sum\limits_{i=1}^{n} \mathop {sup}\limits_{t\geq 1} \{ e^{c^*t^2} \mathbb{I}_{|x_i^T (\beta^* - \tilde \beta)| \geq t/3} \} \leq  e^{c^*} \} \supseteq \{ \mathop {max}\limits_{1\leq i \leq n} |x_i^T (\beta^* - \tilde \beta)| < 1/3\}. \]

It is clear that
\[ \frac{1}{n} \sum\limits_{i=1}^{n} \mathop {sup}\limits_{0 \leq t \leq 1} \{ e^{c^*t^2} \mathbb{I}_{|x_i^T (\beta^* - \tilde \beta)| \geq t/3} \} \leq e^{c^*}, \]
therefore, with probability going to 1, we have
\begin{eqnarray}
\label{eqn:trick}
& &\frac{1}{n} \sum\limits_{i=1}^{n} \mathop {sup}\limits_{t\geq 0} \{ e^{c^*t^2} \mathbb{I}_{|x_i^T (\beta^* - \tilde \beta)| \geq t/3} \} \nonumber \\
&=& \max ( \frac{1}{n} \sum\limits_{i=1}^{n} \mathop {sup}\limits_{0\leq t \leq 1} \{ e^{c^*t^2} \mathbb{I}_{|x_i^T (\beta^* - \tilde \beta)| \geq t/3} \},\  \frac{1}{n} \sum\limits_{i=1}^{n}  \mathop {sup}\limits_{t\geq 1} \{ e^{c^*t^2} \mathbb{I}_{|x_i^T (\beta^* - \tilde \beta)| \geq t/3} \} ) \nonumber \\
& & \leq e^{c^*}.  \nonumber \\
\end{eqnarray}

For the second term in $\eqref{eqn:1121}$, by strong law of large numbers, we have $\overline{\epsilon}\rightarrow 0, \ a.s.$, then
\[ P(|\overline{\epsilon}|\geq 1/3) \rightarrow 0. \]
It is easy to show that
\[ \{\mathop {sup}\limits_{t\geq 1} \{ e^{c^*t^2} \mathbb{I}_{|\overline{\epsilon}| \geq t/3}  \} \supseteq \{ |\overline{\epsilon}| < 1/3 \}.\]
Hence
\begin{equation}
\label{eqn:trick1}
P(\mathop {sup}\limits_{t\geq 1} \{ e^{c^*t^2} \mathbb{I}_{|\overline{\epsilon}| \geq t/3} \} \leq e^{c^*}) \geq P(|\overline{\epsilon}| < 1/3) \rightarrow 1.
\end{equation}
Using the same trick as $\eqref{eqn:trick}$, we have
\begin{equation}
\frac{1}{n} \sum\limits_{i=1}^{n} \mathop {sup}\limits_{t\geq 0} \{e^{c^*t^2} \mathbb{I}_{|\overline{\epsilon}| \geq t/3} \} = \max ( \mathop {sup}\limits_{0\leq t \leq 1} \{e^{c^*t^2} \mathbb{I}_{|\overline{\epsilon}| \geq t/3} \} ,\ \mathop {sup}\limits_{t\geq 1} \{ e^{c^*t^2} \mathbb{I}_{|\overline{\epsilon}| \geq t/3} \} ) \leq e^{c^*}
\end{equation}
holds in probability.

For the third term in $\eqref{eqn:1121}$, we will show that if $c^* = \frac{1}{36\sigma^2}$
\begin{equation}
\label{eqn:finite}
E \mathop {sup}\limits_{t\geq 0} \{e^{c^*t^2} \mathbb{I}_{|\epsilon_1| \geq t/3}\}  =  \int\limits_{0}^{\infty} P(\mathop {sup}\limits_{t\geq 0} \{e^{c^*t^2} \mathbb{I}_{|\epsilon_1| \geq t/3}\} >u) d_u <\infty.
\end{equation}
The above integral can be divided into two parts $[0,e^{9\sigma^2 c^*}]$ and $[e^{9\sigma^2 c^*},\infty]$. And the first part is bounded using $P(\mathop {sup}\limits_{t\geq 0} \{e^{c^*t^2} \mathbb{I}_{|\epsilon_1| \geq t/3}\} >u) \leq 1$ that
\begin{equation}
\label{eqn:part1}
\int\limits_{0}^{e^{9\sigma^2 c^*}} P(\mathop {sup}\limits_{t\geq 0} \{e^{c^*t^2} \mathbb{I}_{|\epsilon_1| \geq t/3}\} >u) d_u \leq e^{9\sigma^2 c^*}.
\end{equation}
For the second part, we have
\begin{eqnarray}
& & P(\mathop {sup}\limits_{t\geq 0} \{e^{c^*t^2} \mathbb{I}_{|\epsilon_1| \geq t/3}\} >u) \nonumber \\
 &= & P(\mathop {sup}\limits_{t\geq 0} \{e^{c^*t^2} \mathbb{I}_{|\epsilon_1| \geq t/3}\} >u, |\epsilon_1| < \sqrt{\frac{\log u}{c^*}}/3 ) \nonumber \\
  & & + P(\mathop {sup}\limits_{t\geq 0} \{e^{c^*t^2} \mathbb{I}_{|\epsilon_1| \geq t/3}\} >u ,|\epsilon_1| \geq \sqrt{\frac{\log u}{c^*}}/3 )  \nonumber \\
 &\leq & P(\mathop {sup}\limits_{t\geq 0} \{e^{c^*t^2} \mathbb{I}_{|\epsilon_1| \geq t/3}\} >u, |\epsilon_1| < \sqrt{\frac{\log u}{c^*}}/3 ) + P(|\epsilon_1| \geq \sqrt{\frac{\log u}{c^*}}/3 ).  \nonumber
\end{eqnarray}
On $\{ |\epsilon_1| < \sqrt{\frac{\log u}{c^*}}/3 \}$, we have
\begin{equation}
e^{c^*t^2} \mathbb{I}_{|\epsilon_1| \geq t/3} \left \{
  \begin{array}{cc}
    \leq u  & if \ t \leq \sqrt{\frac{\log u}{c^*}} \\
    0 & otherwise \\
  \end{array}
\right.
\end{equation}
Hence on $\{ |\epsilon_1| < \sqrt{\frac{\log u}{c^*}}/3 \}$,
\[ \mathop {sup}\limits_{t\geq 0} \{e^{c^*t^2} \mathbb{I}_{|\epsilon_1| \geq t/3}\} \leq u, \]
then
\[ P(\mathop {sup}\limits_{t\geq 0} \{e^{c^*t^2} \mathbb{I}_{|\epsilon_1| \geq t/3}\} >u, |\epsilon_1| < \sqrt{\frac{\log u}{c^*}}/3 ) = 0. \]
Therefore,
\begin{equation}
\label{eqn:tail}
P(\mathop {sup}\limits_{t\geq 0} \{e^{c^*t^2} \mathbb{I}_{|\epsilon_1| \geq t/3}\} >u) \leq P(|\epsilon_1| \geq \sqrt{\frac{\log u}{c^*}}/3 ) \leq 2 e^{-\frac{1}{2\sigma^2} \frac{\log u}{9 c^*}} = \frac{2}{u^2}
\end{equation}
where the second inequality comes from the Gaussian tail bound,
$$P(|\epsilon_1| \geq t ) \leq 2 e^{-\frac{t^2}{2\sigma^2}} \  \forall t \geq \sigma.$$
Therefore
\begin{equation}
\label{eqn:part2}
\int\limits_{e^{9\sigma^2 c^*}}^{\infty} P(\mathop {sup}\limits_{t\geq 0} \{e^{c^*t^2} \mathbb{I}_{|\epsilon_1| \geq t/3}\} >u) d_u \leq \int\limits_{e^{9\sigma^2 c^*}}^{\infty} \frac{2}{u^2} d_u < \infty.
\end{equation}
Combine $\eqref{eqn:part1}$ and $\eqref{eqn:part2}$, we obtain $\eqref{eqn:finite}$.

Again, by strong law of large numbers,
\[ \frac{1}{n} \sum\limits_{i=1}^{n} \mathop {sup}\limits_{t\geq 0} \{ e^{c^*t^2} \mathbb{I}_{|\epsilon_i| \geq t/3} \} \rightarrow E \mathop {sup}\limits_{t\geq 0} \{ e^{c^*t^2} \mathbb{I}_{|\epsilon_i| \geq t/3} \} \ a.s., \]
hence,
\[ \frac{1}{n} \sum\limits_{i=1}^{n} \mathop {sup}\limits_{t\geq 0} \{ e^{c^*t^2} \mathbb{I}_{|\epsilon_i| \geq t/3} \} \leq 2 E \mathop {sup}\limits_{t\geq 0} \{ e^{c^*t^2} \mathbb{I}_{|\epsilon_i| \geq t/3} \}  \ holds \ in\ probability. \]
If we chose $C^*=(2 e^{c^*} + 2 E \mathop {sup}\limits_{t\geq 0} \{ e^{c^*t^2} \mathbb{I}_{|\epsilon_i| \geq t/3} \})$ and $c^* = \frac{1}{36\sigma^2}$, then $\eqref{eqn:subguassionbootstrap}$ holds in probability, which verifies the claim.
\end{proof}

Secondly, we need the following Lemma~\ref{lemma:bootstrap bias}, which provides upper bounds of mean squared error (MSE) of Select+mLS $\tilde \beta^*$ based on the resample $(X,Y^*)$. Let $\sigma_*^2 = C^*/c^*$.
\begin{lemma}
\label{lemma:bootstrap bias}
Under assumptions (a)-(h) and (dd), the following hold
$$E^*||\tilde \beta^* - \tilde \beta||_2^2 \mathbb{I}_{ \hat S^* =  S } = O_p(  \frac{\sigma_*^2}{n} tr(C_{11}^{-1})),$$
$$E^*||\tilde \beta^* - \tilde \beta||_2^2 \mathbb{I}_{ \hat S^* \neq  S } = o_p(e^{-n^{c_2}/4})  $$
where $E^*$ denotes the conditional expectation given the error variables $\{\epsilon_i,i=1,...,n\}$.
\end{lemma}

\begin{proof} By Lemma~\ref{lemma:subgaussian}, we have shown that with probability going to 1, $\epsilon^*_i$ are i.i.d. subgaussian variables, i.e.,
\[ P^*(|\epsilon_i^*| \geq t) \leq C^* e^{-c^*t^2}, \ \forall t\geq 0. \]
Simple calculation yields $E^*(\epsilon_1^*)^2 \leq \sigma_*^2$ and $E^*(\epsilon_1^*)^4 \leq \sigma_*^4$. Conditioned on $\{\hat S =S \}$, we have
\[ ||\tilde \beta^* - \tilde \beta||_2^2 \mathbb{I}_{ \hat S^* =  S } = ||(X_{S}^TX_{S})^{-1}X_{S}^TY^* - \tilde \beta||_2^2 \mathbb{I}_{ \hat S^* =  S } = ||(X_{S}^TX_{S})^{-1}X_{S}^T\epsilon^*||_2^2 \mathbb{I}_{ \hat S^* =  S }. \]
Take expectation, then
\begin{eqnarray}
  E^*||\tilde \beta^* - \tilde \beta||_2^2 \mathbb{I}_{ \hat S^* =  S }  \leq  E^*||(X_{S}^TX_{S})^{-1}X_{S}^T\epsilon^*||_2^2 \leq
  \frac{\sigma_*^2}{n} tr(C_{11}^{-1}). \nonumber
\end{eqnarray}
Since $P(\hat S =S)\rightarrow 1 $, we obtain
\[ E^*||\tilde \beta^* - \tilde \beta||_2^2 \mathbb{I}_{ \hat S^* =  S } = O_p(  \frac{\sigma_*^2}{n} tr(C_{11}^{-1})). \]
For the second part, we also conditioned on $\{\hat S =S \}$. Using the same procedure as proving Theorem~\ref{theorem: asymptotic bias and variance}, we can get
\begin{eqnarray}
 E^* ||\tilde \beta^* - \tilde \beta||_2^2\mathbb{I}_{ \hat S^* \neq  S} & \leq & 8 \sqrt{ P^*(\hat S^* \neq \hat S) }  \{  ||\tilde \beta||_2^2 + \frac{1}{\tau_n^2} \frac{1}{n}||X\tilde \beta||_2^2  +\frac{1}{\tau_n^2} \sigma_*^2 \}.  \nonumber \\
 \label{eqn:a}
\end{eqnarray}

By Theorem~\ref{theorem: asymptotic bias and variance},
\begin{equation}
 E||\tilde \beta - \beta^*||_2^2 \leq \frac{\sigma^2}{n} tr(C_{11}^{-1}) + 8 \sqrt{ P(\hat S \neq S) }  \left\{ ||\beta^*||_2^2 + \frac{1}{\tau_n^2}\frac{1}{n}||X\beta^*||_2^2  + \frac{1}{\tau_n^2}\sigma^2 \right\}.
\end{equation}
As shown in Corollary~\ref{corollary: asymptotic bias and variance}, we have
\[ E||\tilde \beta - \beta^*||_2^2 =  O( \frac{\sigma^2}{\Lambda_{min}} \frac{s}{n}), \]
then $||\tilde \beta- \beta^* ||_2^2 \rightarrow_p 0$ and therefore $||\tilde \beta||_2^2 = ||\beta^* ||_2^2 + o_p(1)$. Moreover,
\begin{equation}
\label{eqn:1usedintheorem3}
\frac{1}{n}||X\tilde \beta-X\beta^*||_2^2 = \frac{1}{n} ||X_S(X_{S}^TX_{S})^{-1}X_{S}^T\epsilon||_2^2 \rightarrow_p 0
\end{equation}
where "$\rightarrow_p 0$" comes from
\[ E \frac{1}{n} ||X_S(X_{S}^TX_{S})^{-1}X_{S}^T\epsilon||_2^2 = \frac{s}{n} \sigma^2 \rightarrow 0, \]
therefore $\frac{1}{n}||X\tilde \beta||_2^2 = \frac{1}{n}||X\beta^*||_2^2 + o_p(1)$. Taking these results back to $\eqref{eqn:a}$ and combining assumption (f), we have
\begin{eqnarray}
E^* ||\tilde \beta^* - \tilde \beta||_2^2\mathbb{I}_{ \hat S^* \neq \hat S} & \leq & 8 o_p(e^{-n^{c_2}/2}) \{  || \beta^*||_2^2 + \frac{1}{\tau_n^2} \frac{1}{n}||X \beta^*||_2^2  + \frac{1}{\tau_n^2} \sigma_*^2 + o_p(1) \} \nonumber \\
&=& o_p(e^{-n^{c_2}/4}) \nonumber
\end{eqnarray}
where the last equality holds since we suppose that $\tau_n \propto \frac{1}{n}$  and that
\[ \frac{1}{n}||X\beta^*||_2^2 =O(n). \]
\end{proof}

Finally, we can prove Theorem~\ref{theorem:bootstraplassovalid} now.

\begin{proof}[Proof of Theorem~\ref{theorem:bootstraplassovalid}] Using the same notations as \cite{BickelFreedman1981a}, let $F$ be the true distribution of $\varepsilon_i$; let $F_n$ be the empirical distribution of $\epsilon_1,...,\epsilon_n$; let $\tilde F_n$ be the empirical distribution of the residuals $\hat \epsilon_1,...,\hat \epsilon_n$; and let $\hat F_n$ be $\tilde F_n$ centered at its mean $\hat \mu$. We first show that the Mallows metric of $G_n$ and $G_n^*$ can be bounded by $\sqrt{s}\cdot d(F,\hat F_n)$.

\begin{lemma}
\label{lemma:theorem31}
Suppose that conditions (a)-(h) and (dd) are satisfied, then
\[ d^2(G_n,G_n^*) \leq 4tr\left\{(\frac{1}{n}X_S^TX_S)^{-1}\right\} d^2(F,\hat F_n)+o_p(1) \leq \frac{4s}{\Lambda_{min}} d^2(F,\hat F_n) +o_p(1). \]
\end{lemma}
\begin{proof} By assumption (f), there exists a set $A_n$ be such that $P(A_n)\rightarrow 1$ and for every $\omega \in A_n$,
\[ P^*(\hat S^* \neq \hat S) = o(e^{-n^{c_2}}). \]
Fix $\omega \in A_n \bigcap \{\hat S = S\}$. By definition of Mallows metric, we have
\begin{eqnarray}
 d^2(G_n,G_n^*) &=&   \mathop {inf}\limits_{\epsilon \sim F, \epsilon^* \sim  \hat F_n} E|| T_n -T_n^* ||_2^2 \nonumber \\
  &=& \mathop {inf}\limits_{\epsilon \sim F, \epsilon^* \sim  \hat F_n} E|| \sqrt{n}(\tilde \beta-\beta^*) - \sqrt{n}(\tilde \beta^*(\omega)-\tilde \beta(\omega)) ||_2^2. \nonumber
\end{eqnarray}
By Lemma 8.1 in \cite{BickelFreedman1981a}, the infimum in Mallows metric can be obtained. Then we can choose pairs $\{\epsilon_i\sim F, \epsilon_i^*\sim \hat F_n,i=1,...,n\}$ which are independent and $E(\epsilon_i - \epsilon_i^*)^2 = d^2(F,\hat F_n)$. Now, Let $A = \{\hat S = S\}$ and $A^* = \{\hat S^* = S\}$, straightforward computation and triangle inequality yield
\begin{eqnarray}
  &  & E|| \sqrt{n}(\tilde \beta-\beta^*) - \sqrt{n}(\tilde \beta^*(\omega)-\tilde \beta(\omega)) ||_2^2 \nonumber \\
 & = & E|| \sqrt{n}(\tilde \beta-\beta^*)\mathbb{I}_{A} +  \sqrt{n}(\tilde \beta-\beta^*)\mathbb{I}_{A^c}  - \sqrt{n}(\tilde \beta^*(\omega)-\tilde \beta(\omega)) \mathbb{I}_{A^*} \nonumber \\
 & & - \sqrt{n} (\beta^*(\omega)-\tilde \beta(\omega)) \mathbb{I}_{(A^*)^c} ||_2^2 \nonumber \\
 &\leq & 2E|| \sqrt{n}(\tilde \beta-\beta^*)\mathbb{I}_{A}- \sqrt{n}(\tilde \beta^*(\omega)-\tilde \beta(\omega)) \mathbb{I}_{A^*}||_2^2 + 4 E|| \sqrt{n}(\tilde \beta-\beta^*)\mathbb{I}_{A^c}||_2^2 \nonumber \\
 & &  +4 E^* || \sqrt{n} (\tilde \beta^*(\omega)-\tilde \beta(\omega)) \mathbb{I}_{(A^*)^c} ||_2^2 \nonumber \\
 &=& 2E|| \sqrt{n}(X_{S}^TX_{S})^{-1}X_{S}^T\epsilon \mathbb{I}_{A}- \sqrt{n}(X_{S}^TX_{S})^{-1}X_{S}^T\epsilon^* \mathbb{I}_{A^*}||_2^2  \nonumber \\
 & &  + 4 E|| \sqrt{n}(\tilde \beta-\beta^*)\mathbb{I}_{A^c}||_2^2 +4 E^* || \sqrt{n} (\tilde \beta^*(\omega)-\tilde \beta(\omega)) \mathbb{I}_{(A^*)^c} ||_2^2 \nonumber
\end{eqnarray}
where $E^*$ is the conditional expectation over $\epsilon^*$ given $\omega$. 

From the proof of Theorem~\ref{theorem: asymptotic bias and variance}, we have
\[ E|| \sqrt{n}(\tilde \beta-\beta^*)\mathbb{I}_{A^c}||_2^2 \leq 8n \sqrt{ P(\hat S \neq S) }  \{ ||\beta^*||_2^2 + \frac{1}{\tau_n^2} \frac{1}{n}||X\beta^*||_2^2  + \frac{1}{\tau_n^2} \sigma^2 \} \rightarrow 0. \]
By Lemma~\ref{lemma:bootstrap bias}, we have
\[ E^* || \sqrt{n} \beta^*(\omega)-\tilde \beta(\omega)) \mathbb{I}_{(A^*)^c} ||_2^2 = o_p(1), \]
therefore
\begin{eqnarray}
\label{eqn:e}
 &  & E|| \sqrt{n}(\tilde \beta-\beta^*) - \sqrt{n}(\tilde \beta^*(\omega)-\tilde \beta(\omega)) ||_2^2 \nonumber \\
 &\leq& 2E|| \sqrt{n}(X_{S}^TX_{S})^{-1}X_{S}^T\epsilon \mathbb{I}_{A}- \sqrt{n}(X_{S}^TX_{S})^{-1}X_{S}^T\epsilon^* \mathbb{I}_{A^*}||_2^2 + o_p(1)  \nonumber \\
 &\leq & 4 E|| \sqrt{n}(X_{S}^TX_{S})^{-1}X_{S}^T\epsilon - \sqrt{n}(X_{S}^TX_{S})^{-1}X_{S}^T\epsilon^* ||_2^2  \nonumber \\
 & & + 4 E|| \sqrt{n}(X_{S}^TX_{S})^{-1}X_{S}^T\epsilon \mathbb{I}_{A^c}- \sqrt{n}(X_{S}^TX_{S})^{-1}X_{S}^T\epsilon^* \mathbb{I}_{(A^*)^c}||_2^2 + o_p(1) \nonumber \\
 &\leq & 4 E|| \sqrt{n}(X_{S}^TX_{S})^{-1}X_{S}^T\epsilon - \sqrt{n}(X_{S}^TX_{S})^{-1}X_{S}^T\epsilon^* ||_2^2  \nonumber \\
 & & +  8 E|| \sqrt{n}(X_{S}^TX_{S})^{-1}X_{S}^T\epsilon \mathbb{I}_{A^c}||_2^2 + 8E^*||\sqrt{n}(X_{S}^TX_{S})^{-1}X_{S}^T\epsilon^* \mathbb{I}_{(A^*)^c}||_2^2 \nonumber \\
 & & + o_p(1).
\end{eqnarray}
Next, we will bound the first three parts in the last inequality respectively. By straightforward computation, we have
\begin{eqnarray}
\label{eqn:d}
 &  & E|| \sqrt{n}(X_{S}^TX_{S})^{-1}X_{S}^T\epsilon - \sqrt{n}(X_{S}^TX_{S})^{-1}X_{S}^T\epsilon^* ||_2^2 \nonumber \\
 &=&  Etr\left\{(\epsilon-\epsilon^*)^T(\sqrt{n}(X_{S}^TX_{S})^{-1}X_{S}^T)^T(\sqrt{n}(X_{S}^TX_{S})^{-1}X_{S}^T)(\epsilon-\epsilon^*)\right\} \nonumber \\
 &=&  tr\left\{(nX_{S}(X_{S}^TX_{S})^{-2}X_{S}^TE(\epsilon-\epsilon^*)(\epsilon-\epsilon^*)^T)\right\} \nonumber \\
 &=&  d^2(F,\hat F_n) tr\left\{nX_{S}(X_{S}^TX_{S})^{-2}X_{S}^T\right\} \nonumber \\
 &=&  d^2(F,\hat F_n) tr\left\{(\frac{1}{n}X_{S}^TX_{S})^{-1}\right\}.
\end{eqnarray}
The penultimate equality is because $E(\epsilon-\epsilon^*)(\epsilon-\epsilon^*)^T = d^2(F,\hat F_n)I$. Then we only need to show that
\begin{equation}
\label{eqn:b}
 E|| \sqrt{n}(X_{S}^TX_{S})^{-1}X_{S}^T\epsilon \mathbb{I}_{A^c}||_2^2 =o(1),
\end{equation}
\begin{equation}
\label{eqn:c}
E^*||\sqrt{n}(X_{S}^TX_{S})^{-1}X_{S}^T\epsilon^* \mathbb{I}_{(A^*)^c}||_2^2=o_p(1).
\end{equation}
It is easy to see that
\begin{eqnarray}
 || \sqrt{n}(X_{S}^TX_{S})^{-1}X_{S}^T\epsilon||_2^2 &=& n\epsilon^TX_{S}(X_{S}^TX_{S})^{-2}X_{S}^T\epsilon   \nonumber \\
 &\leq&  \Lambda_{max}(nX_{S}(X_{S}^TX_{S})^{-2}X_{S}^T) ||\epsilon||_2^2 \nonumber \\
 &\leq&  \Lambda_{max}((\frac{1}{n}X_{S}^TX_{S})^{-1}) ||\epsilon||_2^2 \nonumber \\
 &\leq&  \Lambda_{min}^{-1} ||\epsilon||_2^2. \nonumber
\end{eqnarray}
By Cauchy-Schwarz inequality,
\begin{eqnarray}
 && E|| \sqrt{n}(X_{S}^TX_{S})^{-1}X_{S}^T\epsilon \mathbb{I}_{A^c}||_2^2   \nonumber \\
 &\leq&  \sqrt{E|| \sqrt{n}(X_{S}^TX_{S})^{-1}X_{S}^T\epsilon||_2^4 E\mathbb{I}_{A^c}} \nonumber \\
 &\leq&  \sqrt{\Lambda_{min}^{-2}E||\epsilon||_2^4 P(A^c)  } \nonumber
\end{eqnarray}
In the proof of Theorem~\ref{theorem: asymptotic bias and variance}, we have shown that $E||\epsilon||_2^4=O(n^2)$ (see $\eqref{eqn:useinlemma3}$) and connect with $P(A^c)=P(\hat S \neq S) =o(e^{-n^{c_2}})$, we obtain $\eqref{eqn:b}$. The proof of $\eqref{eqn:c}$ is the same as $\eqref{eqn:b}$, so we omit it.

Combine $\eqref{eqn:d}$, $\eqref{eqn:b}$, $\eqref{eqn:c}$ and $\eqref{eqn:e}$, we have
\[ E|| \sqrt{n}(\tilde \beta-\beta^*) - \sqrt{n}(\tilde \beta^*(\omega)-\tilde \beta(\omega)) ||_2^2 \leq 4 d^2(F,\hat F_n) tr\left\{(\frac{1}{n}X_{S}^TX_{S})^{-1}\right\} + o_p(1). \]
Therefore, we can obtain Lemma~\ref{lemma:theorem31} by
\begin{eqnarray}
 d^2(G_n,G_n^*) &=& \mathop {inf}\limits_{\epsilon \sim F, \epsilon^* \sim  \hat F_n} E|| \sqrt{n}(\tilde \beta-\beta^*) - \sqrt{n}(\tilde \beta^*(\omega)-\tilde \beta(\omega)) ||_2^2 \nonumber \\
 &\leq& E|| \sqrt{n}(\tilde \beta-\beta^*) - \sqrt{n}(\tilde \beta^*(\omega)-\tilde \beta(\omega)) ||_2^2 \nonumber \\
 &\leq & 4d^2(F,\hat F_n) tr\left\{(\frac{1}{n}X_{S}^TX_{S})^{-1}\right\} + o_p(1) \nonumber \\
 &\leq & \frac{4s}{\Lambda_{min}} d^2(F,\hat F_n) +o_p(1) \nonumber
\end{eqnarray}
where the last inequality holds because $(\frac{1}{n}X_{S}^TX_{S})^{-1}$ is a $s$ by $s$ matrix and $\Lambda_{max}((\frac{1}{n}X_{S}^TX_{S})^{-1}) \leq \Lambda_{min}^{-1}$.
\end{proof}

In order to prove Theorem~\ref{theorem:bootstraplassovalid}, we only have to show that
\[ \frac{s}{\Lambda_{min}} d^2(F,\hat F_n)=o_p(1). \]

\begin{lemma}
\label{lemma:theorem32}
Suppose that assumptions (a)-(c), (e) and (dd) are satisfied and that $P(\hat S \neq S) \rightarrow 0$, then
\[ sd^2(\tilde F_n,F_n)=o_p(1). \]
\end{lemma}

\begin{proof} By definition,
\[ d^2(\tilde F_n,F_n) \leq \frac{1}{n} \sum\limits_{i=1}^{n} {(\hat \epsilon_i-\epsilon_i)^2} = \frac{1}{n} ||\hat \epsilon- \epsilon||_2^2. \]
Since $\hat \epsilon = Y-X \tilde \beta$ and $\epsilon = Y-X\beta^*$, we have
\[ \hat \epsilon- \epsilon = X(\beta^*-\tilde \beta). \]
Conditioned on $\{\hat S =S\}$,
\[ \frac{s}{n}||\hat \epsilon- \epsilon||_2^2 = \frac{s}{n}||X_S(X_{S}^TX_{S})^{-1}X_{S}^T\epsilon||_2^2 \rightarrow_p 0 \]
because
\[ E\frac{s}{n}||X_S(X_{S}^TX_{S})^{-1}X_{S}^T\epsilon||_2^2 =  \frac{s}{n} tr\left\{X_S(X_{S}^TX_{S})^{-1}X_{S}^T\right\}\sigma^2= \frac{s^2}{n}\sigma^2 \rightarrow 0. \]
Therefore,
\[ sd^2(\tilde F_n,F_n)=o_p(1). \]
\end{proof}

\begin{lemma}
\label{lemma:theorem33}
Suppose that assumptions (a)-(c), (e) and (dd) are satisfied and that $P(\hat S \neq S) \rightarrow 0$, then
\[ s d^2(\hat F_n,F_n) =  o_p(1). \]
\end{lemma}

\begin{proof} Application of Lemma 8.8 in \cite{BickelFreedman1981a}, shows that for random variables $U$ and $V$ with finite second moment,
\[ d^2(U,V) = d^2(U-EU,V-EV)+||EU-EV||_2^2. \]
Therefore, if let $\overline{F_n}$ be the empirical distribution of $\epsilon_1,...,\epsilon_n$ centered at its mean $\overline{\epsilon} =\frac{1}{n} \sum\limits_{i=1}^{n} \epsilon_i$, we have
\[ d^2(\hat F_n,F_n) = d^2(\hat F_n,\overline{F_n}) +  ( \frac{1}{n}  \sum\limits_{i=1}^{n} \epsilon_i)^2, \]
\[ d^2(\tilde F_n,F_n) =  d^2(\hat F_n,\overline{F_n}) + ( \frac{1}{n} \sum\limits_{i=1}^{n} \epsilon_i- \frac{1}{n} \sum\limits_{i=1}^{n} \hat \epsilon_i)^2. \]
Connecting the above two equalities,
\begin{eqnarray}
d^2(\hat F_n,F_n) &=& d^2(\tilde F_n,F_n)  + (\frac{1}{n} \sum\limits_{i=1}^{n} {\epsilon_i})^2 - (\frac{1}{n} \sum\limits_{i=1}^{n} {\{\hat \epsilon_i-\epsilon_i\}})^2 \nonumber \\
&\leq & d^2(\tilde F_n,F_n)  + (\frac{1}{n} \sum\limits_{i=1}^{n} {\epsilon_i})^2. \nonumber
\end{eqnarray}
Now use the fact $ sE (\frac{1}{n} \sum\limits_{i=1}^{n} {\epsilon_i})^2 = \frac{s\sigma^2}{n} \rightarrow 0$ and the previous Lemma~\ref{lemma:theorem32}, we obtain Lemma~\ref{lemma:theorem33}.
\end{proof}

Since $d$ is a metric,
\begin{equation}
\label{eqn:metric}
\frac{1}{2} d^2(F,\hat F_n) \leq d^2(F,F_n) + d^2(F_n,\hat F_n).
\end{equation}
We still need to bound $d^2(F,F_n)$. Denote $\phi$ and $\Phi$ the density and distribution functions of standard normal distribution $N(0,1)$ respectively. To control $d^2(F,F_n)$, we use the following result obtained by \cite{delBarrio2000}: let $\rightarrow_\omega$ denote weak convergence,
\begin{lemma}[del Barrio et al. (2000)]
\label{lemma:theorem34}
Let $\epsilon_i,i=1,..,n$ be a sequence of i.i.d. normal random variables with mean 0 and variance $\sigma^2$. Then
\[ n(\frac{d^2(F,F_n)}{\sigma^2}-a_n) \rightarrow_\omega -\frac{3}{2} + \sum\limits_{j=3}^{\infty} \frac{Z_j^2-1}{j} \]
where $\{ Z_j,j=3,..,\infty \}$ are a sequence of independent $N(0,1)$ random variables and
\[ a_n = \frac{1}{n} \int_{\frac{1}{n+1}}^{\frac{n}{n+1}} \frac {t(1-t)}{[\phi(\Phi^{-1}(t))]^2} dt. \]
\end{lemma}
In fact we have (see \cite{Bickelvan1978})
\[ \int_{\frac{1}{n}}^{1-\frac{1}{n}} \frac {t(1-t)}{[\phi(\Phi^{-1}(t))]^2} d_t = \log \log n + \log 2 + \gamma + o(1) \]
where $\gamma = \lim_{k\rightarrow \infty} ( \sum\limits_{i=1}^{n} j^{-1} -\log k )$ is Euler's constant. Then, we have
\[ d^2(F,F_n) = O_p(\frac{1}{n})+ \sigma^2 \frac{\log\log n}{n} = o_p(\frac{1}{s}), \]
which together with Lemma~\ref{lemma:theorem31}, Lemma~\ref{lemma:theorem33} and inequality $\eqref{eqn:metric}$ complete the proof.

\end{proof}
\ \\

\begin{proof}[Proof of Lemma~\ref{lemma:bootstrap selection consistency}] Let $\hat F_n$ be the empirical distribution of the centered residuals $\hat \epsilon_1-\hat \mu,...,\hat \epsilon_n -\hat \mu$. For the stared data $(X,Y^*)$, we have
\[ Y^*=X\tilde \beta + \epsilon^*, \ \epsilon_i^* \sim \hat F_n  \ i.i.d.. \]

We only need to verify: the stared version (replacing $\beta^*$ and $\epsilon$ by $\tilde \beta$ and $\epsilon^*$ respectively) of conditions (a)-(d), (i), (j) and the Irrepresentable Condition $\eqref{eqn:IC}$ hold in probability. Given $\{\hat S =S\}$, these conditions have the following forms:

$(Irrepresentable\ Condition^*)$: there exists a positive constant vector $\eta$, such that
\begin{equation}
 |C_{21}C_{11}^{-1}sign(\tilde \beta_S)|\leq \mathbf{1}-\eta.
\end{equation}

$(a^*)$\footnote{By Remark 2.1, subgaussian assumption of $\epsilon_i^*$ ensures model selection consistency of Lasso}: $\epsilon_i^*$ are i.i.d. subgaussian random variables. That is, there exists constant $C^*$, $c^*>0$ such that
\[ P^*(|\epsilon_i^*|\geq t)\leq C^* e^{-c^*t^2}, \ \forall t\geq 0. \]

$(b^*)$: Suppose that the predictors are standardized, i.e.
\begin{equation}
 \sum\limits_{i=1}^{n} {x_{ij}} = 0 \ and \ \frac{1}{n} \sum\limits_{i=1}^{n} {x_{ij}^2}=1, \ j=1,...,p.
 \label{eqn:standardized}
\end{equation}

$(c^*)$: there exists an constant $\Lambda_{min}>0$ such that
\begin{equation}
 \Lambda_{min} (C_{11})  \geq \Lambda_{min}.
\end{equation}

$(d^*)$:  let $s^*=|\hat S|$, there exists $0\leq c_1<1$ and $0<c_2<1-c_1$
\begin{equation}
 s^*=s^*_n=O(n^{c_1}) \ , \ p=p_n = O(e^{n^{c_2}}).
\end{equation}

$(i^*)$: there exists constant $c_1+c_2<c_3\leq 1$ and $M>0$ so that
\begin{equation}
 n^{\frac{1-c_3}{2}} \mathop {min}\limits_{1\leq i \leq s} |\tilde \beta_i| \geq M.
\end{equation}

$(j^*)$: $\lambda_n \propto n^{\frac{1+c_4}{2}}$ with $c_2<c_4<c_3-c_1$.
\ \\

Clearly, conditions $(b^*)$, $(c^*)$ and $(j^*)$ hold because they are not relative to $\tilde \beta$ and $\epsilon^*$. By Lemma~\ref{lemma:subgaussian}, condition (a*) holds. $(d^*)$ is satisfied since $s^*=|\hat S| =s$. By Corollary~\ref{corollary: asymptotic normality} (asymptotic normality), $\tilde \beta$ is $\sqrt{n}$-consistent, then condition $(i^*)$ holds in probability. The sign-consistency of $\tilde \beta$ (Lemma~\ref{lemma: model selection}) ensures condition $(Irrepresentable\ Condition^*)$.

\end{proof}

\end{document}